\documentclass[12pt]{amsart}       
\usepackage{txfonts}
\usepackage{amssymb}
\usepackage{eucal}
\usepackage{graphicx}
\usepackage{amsmath}
\usepackage{amscd}
\usepackage[all]{xy}           
\usepackage{amsfonts,latexsym}
\usepackage{xspace}
\usepackage{epsfig}

\usepackage{epstopdf}
\usepackage{subfig}
\usepackage{enumitem}
\usepackage{amsmath,amssymb,amsthm}

\usepackage{tikz-cd}

\usepackage{float}
\usepackage{color}
\usepackage{fancybox}
\usepackage{colordvi}
\usepackage{multicol}
\usepackage{colordvi}
\usepackage{tikz}
\usepackage{wasysym}
\usepackage[active]{srcltx} 
\ifpdf
  \usepackage[colorlinks,final,backref=page,hyperindex]{hyperref}
\else
  \usepackage[colorlinks,final,backref=page,hyperindex,hypertex]{hyperref}
\fi

\newcommand{\nc}{\newcommand}
\newcommand{\delete}[1]{}

\nc{\mlabel}[1]{\label{#1}}  
\nc{\mcite}[1]{\cite{#1}}  
\nc{\mref}[1]{\ref{#1}}  
\nc{\meqref}[1]{\eqref{#1}}  
\nc{\mbibitem}[1]{\bibitem{#1}} 

\delete{
\nc{\mlabel}[1]{\label{#1}  
{\hfill \hspace{1cm}{\small\tt{{\ }\hfill(#1)}}}}
\nc{\mcite}[1]{\cite{#1}{\small{\tt{{\ }(#1)}}}}  
\nc{\mref}[1]{\ref{#1}{{\tt{{\ }(#1)}}}}  
\nc{\meqref}[1]{\eqref{#1}{{\tt{{\ }(#1)}}}}  
\nc{\mbibitem}[1]{\bibitem[\bf #1]{#1}} 
}

\makeatletter

\newcommand{\Rmnum}[1]{\expandafter\@slowromancap\romannumeral #1@}
\makeatother

\newtheorem{theorem}{Theorem}[section]
\newtheorem{prop}[theorem]{Proposition}
\newtheorem{lemma}[theorem]{Lemma}

\theoremstyle{definition}
\newtheorem{defn}[theorem]{Definition}
\newtheorem{prop-def}{Proposition-Definition}[section]

\newtheorem{remark}[theorem]{Remark}

\newtheorem{conjecture}[theorem]{Conjecture}

\newtheorem{tempex}[theorem]{Example}
\newtheorem{tempexs}[theorem]{Examples}
\newtheorem{temprmk}[theorem]{Remark}
\newtheorem{tempexer}{Exercise}[section]
\newenvironment{exam}{\begin{tempex}\rm}{\end{tempex}}

\nc{\vsa}{\vspace{-.1cm}} \nc{\vsb}{\vspace{-.2cm}}
\nc{\vsc}{\vspace{-.3cm}} \nc{\vsd}{\vspace{-.4cm}}
\nc{\vse}{\vspace{-.5cm}}

\nc{\Irr}{\mathrm{Irr}}
\nc{\ncrbw}{\calr}  
\nc{\NS}{U_{NS}}
\nc{\FN}{F_{\mathrm Nij}}
\nc{\dfgen}{V} \nc{\dfrel}{R}
\nc{\dfgenb}{\vec{v}} \nc{\dfrelb}{\vec{r}}
\nc{\dfgene}{v} \nc{\dfrele}{r}
\nc{\dfop}{\odot}
\nc{\dfoa}{\dfop^{(1)}} \nc{\dfob}{\dfop^{(2)}}
\nc{\dfoc}{\dfop^{(3)}} \nc{\dfod}{\dfop^{(4)}}
\nc{\mapm}[1]{\lfloor\!|{#1}|\!\rfloor}
\nc{\cmapm}[1]{\frakC(#1)}
\nc{\red}{\mathrm{Red}}
\nc{\cm}{C}
\nc{\supp}{\mathrm{Supp}}
\nc{\lex}{\mathrm{lex}}

\nc{\disp}[1]{\displaystyle{#1}}
\nc{\bin}[2]{ (_{\stackrel{\scs{#1}}{\scs{#2}}})}  
\nc{\bs}{\bar{S}} \nc{\ep}{\epsilon}
\nc{\dbigcup}{\stackrel{\bullet}{\bigcup}}
\nc{\la}{\longrightarrow} \nc{\cprod}{\ast} \nc{\rar}{\rightarrow}
\nc{\dar}{\downarrow} \nc{\labeq}[1]{\stackrel{#1}{=}}
\nc{\dap}[1]{\downarrow \rlap{$\scriptstyle{#1}$}}
\nc{\uap}[1]{\uparrow \rlap{$\scriptstyle{#1}$}}
\nc{\defeq}{\stackrel{\rm def}{=}} \nc{\dis}[1]{\displaystyle{#1}}
\nc{\dotcup}{\ \displaystyle{\bigcup^\bullet}\ }
\nc{\sdotcup}{\tiny{ \displaystyle{\bigcup^\bullet}\ }}
\nc{\fe}{\'{e}}
\nc{\hcm}{\ \hat{,}\ } \nc{\hcirc}{\hat{\circ}}
\nc{\hts}{\hat{\shpr}} \nc{\lts}{\stackrel{\leftarrow}{\shpr}}
\nc{\denshpr}{\den{\shpr}}
\nc{\rts}{\stackrel{\rightarrow}{\shpr}} \nc{\lleft}{[}
\nc{\lright}{]} \nc{\uni}[1]{\tilde{#1}} \nc{\free}[1]{\bar{#1}}
\nc{\freea}[1]{\tilde{#1}} \nc{\freev}[1]{\hat{#1}}
\nc{\dt}[1]{\hat{#1}}
\nc{\wor}[1]{\check{#1}}
\nc{\intg}[1]{F_C(#1)}
\nc{\den}[1]{\check{#1}} \nc{\lrpa}{\wr} \nc{\mprod}{\pm}
\nc{\dprod}{\ast_P} \nc{\curlyl}{\left \{ \begin{array}{c} {} \\
{} \end{array}
    \right .  \!\!\!\!\!\!\!}
\nc{\curlyr}{ \!\!\!\!\!\!\!
    \left . \begin{array}{c} {} \\ {} \end{array}
    \right \} }
\nc{\longmid}{\left | \begin{array}{c} {} \\ {} \end{array}
    \right . \!\!\!\!\!\!\!}
\nc{\lin}{\call} \nc{\ot}{\otimes}
\nc{\ora}[1]{\stackrel{#1}{\rar}}
\nc{\ola}[1]{\stackrel{#1}{\la}}
\nc{\scs}[1]{\scriptstyle{#1}} \nc{\mrm}[1]{{\rm #1}}
\nc{\margin}[1]{\marginpar{\rm #1}}   
\nc{\dirlim}{\displaystyle{\lim_{\longrightarrow}}\,}
\nc{\invlim}{\displaystyle{\lim_{\longleftarrow}}\,}
\nc{\mvp}{\vspace{0.5cm}}
\nc{\mult}{m}       
\nc{\svp}{\vspace{2cm}} \nc{\vp}{\vspace{8cm}}
\nc{\proofbegin}{\noindent{\bf Proof: }}
\nc{\proofend}{$\blacksquare$ \vspace{0.5cm}}
\nc{\sha}{{\mbox{\cyr X}}}  
\nc{\ncsha}{{\mbox{\cyr X}^{\mathrm NC}}}
\newfont{\scyr}{wncyr10 scaled 550}
\nc{\ssha}{\mbox{\bf \scyr X}}
\nc{\ncshao}{{\mbox{\cyr X}^{\mathrm NC,\,0}}}
\nc{\shpr}{\diamond}    
\nc{\shprc}{\shpr_c}
\nc{\shpro}{\diamond^0}    
\nc{\shpru}{\check{\diamond}} \nc{\spr}{\cdot}
\nc{\catpr}{\diamond_l} \nc{\rcatpr}{\diamond_r}
\nc{\lapr}{\diamond_a} \nc{\lepr}{\diamond_e} \nc{\sprod}{\bullet}
\nc{\un}{u}                 
\nc{\vep}{\varepsilon} \nc{\labs}{\mid\!} \nc{\rabs}{\!\mid}
\nc{\hsha}{\widehat{\sha}} \nc{\psha}{\sha^{+}} \nc{\tsha}{\tilde{\sha}}
\nc{\lsha}{\stackrel{\leftarrow}{\sha}}
\nc{\rsha}{\stackrel{\rightarrow}{\sha}} \nc{\lc}{\lfloor}
\nc{\rc}{\rfloor} \nc{\sqmon}[1]{\langle #1\rangle}
\nc{\altx}{\Lambda} \nc{\vecT}{\vec{T}} \nc{\piword}{{\mathfrak P}}
\nc{\lbar}[1]{\overline{#1}}
\nc{\dep}{\mathrm{dep}}

\nc{\mmbox}[1]{\mbox{\ #1\ }}
\nc{\ayb}{\mrm{AYB}} \nc{\mayb}{\mrm{mAYB}} \nc{\cyb}{\mrm{cyb}}
\nc{\ann}{\mrm{ann}} \nc{\Aut}{\mrm{Aut}} \nc{\cabqr}{\mrm{CABQR
}} \nc{\can}{\mrm{can}} \nc{\colim}{\mrm{colim}}
\nc{\Cont}{\mrm{Cont}} \nc{\rchar}{\mrm{char}}
\nc{\cok}{\mrm{coker}} \nc{\dtf}{{R-{\rm tf}}} \nc{\dtor}{{R-{\rm
tor}}}

\nc{\Div}{{\mrm Div}} \nc{\End}{\mrm{End}} \nc{\Ext}{\mrm{Ext}}
\nc{\FG}{\mrm{FG}} \nc{\Fil}{\mrm{Fil}} \nc{\Frob}{\mrm{Frob}}
\nc{\Gal}{\mrm{Gal}} \nc{\GL}{\mrm{GL}} \nc{\Hom}{\mrm{Hom}}
\nc{\hsr}{\mrm{H}} \nc{\hpol}{\mrm{HP}} \nc{\id}{\mrm{id}} \nc{\Id}{\mathrm{Id}}  \nc{\ID}{\mathrm{ID}}
\nc{\im}{\mrm{im}} \nc{\incl}{\mrm{incl}} \nc{\Loday}{\mrm{ABQR}\
} \nc{\length}{\mrm{length}} \nc{\LR}{\mrm{LR}} \nc{\mchar}{\rm
char} \nc{\pmchar}{\partial\mchar} \nc{\map}{\mrm{Map}}
\nc{\MS}{\mrm{MS}} \nc{\OS}{\mrm{OS}} \nc{\NC}{\mrm{NC}}
\nc{\rba}{\rm{Rota-Baxter algebra}\xspace}
\nc{\rbas}{\rm{Rota-Baxter algebras}\xspace}
\nc{\rbw}{\ncrbw}
\nc{\rbws}{\rm{RBWs}\xspace}
\nc{\rbadj}{\rm{RB}\xspace}
\nc{\mpart}{\mrm{part}} \nc{\ql}{{\QQ_\ell}} \nc{\qp}{{\QQ_p}}
\nc{\rank}{\mrm{rank}} \nc{\rcot}{\mrm{cot}} \nc{\rdef}{\mrm{def}}
\nc{\rdiv}{{\rm div}} \nc{\rtf}{{\rm tf}} \nc{\rtor}{{\rm tor}}
\nc{\res}{\mrm{res}} \nc{\SL}{\mrm{SL}} \nc{\Spec}{\mrm{Spec}}
\nc{\tor}{\mrm{tor}} \nc{\Tr}{\mrm{Tr}}
\nc{\mtr}{\mrm{tr}}

\nc{\ab}{\mathbf{Ab}} \nc{\Alg}{\mathbf{Alg}}
\nc{\Bax}{\mathbf{CRB}} \nc{\Algo}{\mathbf{Alg}^0}
\nc{\cRB}{\mathbf{CRB}} \nc{\cRBo}{\mathbf{CRB}^0}
\nc{\RBo}{\mathbf{RB}^0} \nc{\BRB}{\mathbf{RB}}
\nc{\Dend}{\mathbf{DD}} \nc{\bfk}{{\bf k}} \nc{\bfone}{{\bf 1}}
\nc{\base}[1]{{a_{#1}}} \nc{\Cat}{\mathbf{Cat}}
 \nc{\DN}{\mathbf{DN}}
\nc{\NA}{\mathbf{NA}}
\nc{\SDN}{\mathbf{SDN}}
\nc{\Diff}{\mathbf{Diff}} \nc{\gap}{\marginpar{\bf
Incomplete}\noindent{\bf Incomplete!!}
    \svp}
\nc{\FMod}{\mathbf{FMod}} \nc{\Int}{\mathbf{Int}}
\nc{\Mon}{\mathbf{Mon}}
\nc{\RB}{\mathbf{RB}} \nc{\remarks}{\noindent{\bf Remarks: }}
\nc{\Rep}{\mathbf{Rep}} \nc{\Rings}{\mathbf{Rings}}
\nc{\Sets}{\mathbf{Sets}} \nc{\bfx}{\mathbf{x}}
\nc{\BA}{{\Bbb A}} \nc{\CC}{{\Bbb C}} \nc{\DD}{{\Bbb D}}
\nc{\EE}{{\Bbb E}} \nc{\FF}{{\Bbb F}} \nc{\GG}{{\Bbb G}}
\nc{\HH}{{\Bbb H}} \nc{\LL}{{\Bbb L}} \nc{\NN}{{\Bbb N}}
\nc{\QQ}{{\Bbb Q}} \nc{\RR}{{\Bbb R}} \nc{\TT}{{\Bbb T}}
\nc{\VV}{{\Bbb V}} \nc{\ZZ}{{\Bbb Z}}

\nc{\cala}{{\mathcal A}} \nc{\calb}{{\mathcal B}}
\nc{\calc}{{\mathcal C}}
\nc{\cald}{{\mathcal D}} \nc{\cale}{{\mathcal E}}
\nc{\calf}{{\mathcal F}} \nc{\calg}{{\mathcal G}}
\nc{\calh}{{\mathcal H}} \nc{\cali}{{\mathcal I}}
\nc{\calj}{{\mathcal J}} \nc{\call}{{\mathcal L}}
\nc{\calm}{{\mathcal M}} \nc{\caln}{{\mathcal N}}
\nc{\calo}{{\mathcal O}} \nc{\calp}{{\mathcal P}}
\nc{\calr}{{\mathcal R}} \nc{\cals}{{\mathcal S}} \nc{\calt}{{\mathcal T}}
\nc{\calw}{{\mathcal W}} \nc{\calx}{{\mathcal X}} \nc{\caly}{{\mathcal Y}} \nc{\calz}{{\mathcal Z}}
\nc{\CA}{\mathcal{A}}

\nc{\frakA}{{\mathfrak A}}
\nc{\fraka}{{\mathfrak a}}
\nc{\frakB}{{\mathfrak B}}
\nc{\frakb}{{\mathfrak b}}
\nc{\frakC}{{\mathfrak C}}
\nc{\frakd}{{\mathfrak d}}
\nc{\frakF}{{\mathfrak F}}
\nc{\frakg}{{\mathfrak g}}
\nc{\frakm}{{\mathfrak m}}
\nc{\frakM}{{\mathfrak M}}
\nc{\frakMo}{{\mathfrak M}^0}
\nc{\frakP}{{\mathfrak P}}
\nc{\frakp}{{\mathfrak p}}
\nc{\frakS}{{\mathfrak S}}
\nc{\frakSo}{{\mathfrak S}^0}
\nc{\fraks}{{\mathfrak s}}
\nc{\os}{\overline{\fraks}}
\nc{\frakT}{{\mathfrak T}}
\nc{\frakTo}{{\mathfrak T}^0}
\nc{\oT}{\overline{T}}
\nc{\frakX}{{\mathfrak X}}
\nc{\frakXo}{{\mathfrak X}^0}
\nc{\frakx}{{\mathbf x}}
\nc{\frakTx}{\frakT}      
\nc{\frakTa}{\frakT^a}        
\nc{\frakTxo}{\frakTx^0}   
\nc{\caltao}{\calt^{a,0}}   
\nc{\ox}{\overline{\frakx}} \nc{\fraky}{{\mathfrak y}}
\nc{\frakz}{{\mathfrak z}} \nc{\oX}{\overline{X}} \font\cyr=wncyr10

\nc{\tred}[1]{\textcolor{red}{#1}} \nc{\tgreen}[1]{\textcolor{green}{#1}}
\nc{\tblue}[1]{\textcolor{blue}{#1}} \nc{\tpurple}[1]{\textcolor{purple}{#1}}

\nc{\li}[1]{\tpurple{\underline{Li:}#1 }}
\nc{\liadd}[1]{\tpurple{#1}}
\nc{\xing}[1]{\tblue{\underline{Xing:}#1 }}
\nc{\YZ}[1]{\tred{\underline{Yaozhou:} #1}}

\nc{\deleted}[1]{\delete{#1}}
\nc{\astarrow}{\overset{\raisebox{-2pt}{{\scriptsize $\ast$}}}{\rightarrow}}\nc{\tvarrow}[3]{#1\overset{(t,v)}{\longrightarrow}_{#3} #2}

\nc{\fg}[2]{\tilde{\mathfrak{g}}_{#1}^{(#2)}}
\nc{\De}[2]{\Delta^{#1}(#2)}
\nc{\bott}{b}
\nc{\down}[2]{{\rm down}_{#1}(#2)}
\nc{\Down}[3]{{\rm down}_{#1}^{#2}(#3)}
\nc{\uup}[2]{{\rm up}_{#1}(#2)}
\nc{\dkmu}{\Delta^k(\mu)} \nc{\dkl}{\Delta^k(\lambda)}
\nc{\pk}{{\rm P}^k}
\nc{\pkl}{{\rm P}_{\ell}^{k}}
\nc{\tpkl}{\tilde{{\rm P}}_{\ell}^{k}}
\nc{\spkl}{{\rm sP}_{\ell}^{k}}
\nc{\lmx}{\mu} \nc{\lsum}{\nu}
\nc{\ddh}{\bar{z}}

\begin{document}
\title[Two conjectures of closed $k$-Schur Katalan functions]{Alternating dual Pieri rule conjecture and $k$-branching conjecture of closed $k$-Schur Katalan functions}
%
\author{Yaozhou Fang}
\address{School of Mathematics and Statistics, Lanzhou University
Lanzhou, 730000, China
}
\email{fangyzh21@lzu.edu.cn}

\author{Xing Gao$^*$}\thanks{*Corresponding author}
\address{School of Mathematics and Statistics, Lanzhou University,
Lanzhou, 730000, China; Gansu Provincial Research Center for Basic Disciplines of Mathematics and Statistics, Lanzhou, 730070, China
}
\email{gaoxing@lzu.edu.cn}

%

\date{\today}
\begin{abstract}
For closed $k$-Schur Katalan functions $\fg{\lambda}{k}$ with $k$ a positive integer and $\lambda$ a $k$-bounded partition, Blasiak, Morse and Seelinger proposed the alternating dual Pieri rule conjecture and the $k$-branching conjecture. In the present paper, we positively prove the first one for large enough $k$ and for strictly decreasing partitions $\lambda$ respectively, as well as the second one for strictly decreasing partitions $\lambda$.
\end{abstract}

\makeatletter
\@namedef{subjclassname@2020}{\textup{2020} Mathematics Subject Classification}
\makeatother
\subjclass[2020]{
05E05, 
05E10, 
14N15, 
}

\keywords{$k$-Schur function, Katalan function, closed $k$-Schur Katalan function, positivity}

\maketitle

\tableofcontents

\setcounter{section}{0}

\allowdisplaybreaks

\section{Introduction}
In this paper, we focus on the alternating dual Pieri rule conjecture and the $k$-branching conjecture of closed $k$-Schur Katalan functions $\fg{\lambda}{k}$, posed by Blasiak, Morse and Seelinger in~\cite{BMS}.
In details, we prove the first one for large enough $k$ and for strictly decreasing partitions $\lambda$ respectively, and the second one for strictly decreasing partitions $\lambda$.
Some usual notations are listed in page four below.

\subsection{From $k$-Schur functions to closed $k$-Schur Katalan functions}
Lapointe, Lascoux and Morse~\cite{LLM} proposed the concept of $k$-Schur function, as a generalization of the Schur function,
to study a refinement of Macdonald's positivity conjecture. The research of $k$-Schur functions has yielded fruitful outcomes~\cite{BMPS20,Lam,LLMSSZ,LLMS,LS,LLM,LM03,LM05}. In particular, $k$-Schur functions form a linear basis of the Hopf algbra $\Lambda_{(k)}:= \QQ\{h_1,\ldots,h_k\}$ of symmetric functions~\cite{LLM}.
Geometrically, Lam~\cite{Lam} showed that $k$-Schur functions are Schubert classes for $H_*({\rm Gr})$ under the isomorphism studied by Bott~\cite{Bot}, where ${\rm Gr}$ is the affine Grassmannian. Later, based on the above results, Lam and Shimozono~\cite{LS} exposed an isomorphism from localizations of $\mathcal{QH}^*(F\ell_{k+1})$ to localizations of $H_*({\rm Gr})$, where $\mathcal{QH}^*(F\ell_{k+1})$ is the quantum cohomology ring of the complete flag variety $F\ell_{k+1}$. In that paper, Lam and Shimozono showed that $k$-Schur functions are the core of the image of quantum Schubert polynomials studied by Fomin, Gelfand and Postnikov~\cite{FGP} under the isomorphism in~\cite{LS}. This conclusion established a connection between two very different objects---$k$-Schur functions and quantum Schubert polynomials, and further made properties of both sides interchangeable.

Recently, Ikeda, Iwao and Maeno~\mcite{IIM} learned a $K$-theoretic version $\Phi$ of the isomorphism in~\cite{LS}, from localizations of $\mathcal{QK}(F\ell_{k+1})$ to localizations of $K_*({\rm Gr})$. Here $\mathcal{QK}(F\ell_{k+1})$ is the quantum $K$-theory ring of $F\ell_{k+1}$ studied by Givental and Lee~\mcite{GL}, and $K_*({\rm Gr})$ is the $K$-homology of ${\rm Gr}$ isomorphic to $\Lambda_{(k)}$ as Hopf algebras~\cite{LSS}.
\begin{equation*}
\begin{tikzcd}[column sep=scriptsize, row sep=scriptsize]
\mathcal{QH}^*(F\ell_{k+1})_{{\rm loc}} \arrow{rr}{\text{\cite{LS}}}[swap]{\backsimeq} \arrow{dd}[swap]{\text{$K$-theoretic}}{\text{version}} &&H_*({\rm Gr})_{{\rm loc}} \arrow{dd}[swap]{\text{$K$-theoretic}}{\text{version}}\\
\\
\mathcal{QK}(F\ell_{k+1})_{{\rm loc}} \arrow{rr}{\text{\cite{IIM}}~\Phi}[swap]{\backsimeq} &&K_*({\rm Gr})_{{\rm loc}}
\end{tikzcd}
\end{equation*}

Also, Lenart and Maeno~\mcite{LeMa} exposed a quantum $K$-theoretic version of quantum Schubert polynomials, named quantum Grothendieck polynomials $\mathfrak{B}_{w}^{Q}\in\mathcal{QK}(F\ell_{k+1})$ with $w$ in the symmetric group $S_{k+1}$.
For a better study of the image $\Phi(\mathfrak{B}_{w}^{Q})$, Blasiak, Morse and Seelinger~\mcite{BMS} initiated a family of inhomogenous symmetric functions, called {\bf Katalan functions} (see Definition~\mref{defn:KataFunc}).
There Blasiak et al. also constructed a special kind of Katalan functions $\tilde{\mathfrak{g}}_{\lambda}^{(k)}$, named {\bf closed $k$-Schur Katalan functions} (see Definition~\mref{defn:cKataFunc}), which form a linear basis of $\Lambda_{(k)}$. Later, Ikeda, Iwao and Naito~\mcite{IIN} obtained that closed $k$-Schur Katalan functions $\fg{\lambda}{k}$ are exactly the core of the image of quantum Grothendieck polynomials under $\Phi$.

Although the above work has been completed, there are still some conjectures remaining unsolved about closed $k$-Schur Katalan functions~\cite[Conjecture~2.12]{BMS}.

\subsection{Alternating dual Pieri rule conjecture and $k$-branching conjecture}
Based on closed $k$-Schur Katalan functions, Blasiak et al. proposed the following two conjectures, out of the six conjectures in~\cite[Conjecture~2.12]{BMS}.

\begin{conjecture}
Let $\lambda\in\pkl$ and $m\in\ZZ_{\geq0}$.
\begin{enumerate}
\item (alternating dual Pieri rule)
The coefficients $c_{\lambda\mu}$ in
\begin{equation*}
G_{1^m}^\perp\fg{\lambda}{k} = \sum_{\mu\in{\rm P}}c_{\lambda\mu}\fg{\mu}{k}
\end{equation*}
satisfy $(-1)^{|\lambda|-|\mu|-m}c_{\lambda\mu}\in\ZZ_{\geq0}$.
\mlabel{it:aimalt}

\item ($k$-branching)
The coefficients $a_{\lambda\mu}$ in
\begin{equation*}
\fg{\lambda}{k} = \sum_{\mu\in{\rm P}}a_{\lambda\mu}\fg{\mu}{k+1}
\end{equation*}
satisfy $(-1)^{|\lambda|-|\mu|} a_{\lambda\mu}\in\ZZ_{\geq0}$.
\mlabel{it:aimk}
\end{enumerate}
\mlabel{conj:aim}
\end{conjecture}
Here $G_{1^m}$ with $m\in\ZZ_{\geq0}$ is a special class of {\bf stable Grothendieck polynomials} exposed in~\cite{FK94,FK96,Las90}, and has the form
\begin{equation}
G_{1^m} = \sum_{i\geq 0}(-1)^i\binom{m-1+i}{m-1}e_{m+i}.
\mlabel{eq:G}
\end{equation}

It is worth mentioning that Ikeda, Iwao and Naito~\mcite{IIN} proved other two of the  six conjectures listed in~\cite[Conjecture~2.12]{BMS}.
The $k$-branching conjecture is an important positivity property of closed $k$-Schur Katalan functions.
Indeed, there is a series of natural filters
$$\Lambda_{(1)}\subset\Lambda_{(2)}\subset\cdots$$
by the definition of $\Lambda_{(k)}$. The validness of $k$-branching conjecture
elegantly describes the process of embedding from $\Lambda_{(k)}$ to $\Lambda_{(k+1)}$.
The alternating dual Pieri rule is a sufficient condition for the $k$-branching property~\cite[Remark~2.13]{BMS}.

\subsection{Main results and outline of the present paper}
In this paper, we prove respectively that Conjecture~\mref{conj:aim}~\mref{it:aimalt} is valid for sufficiently large enough positive integers $k$ and for strictly decreasing partitions.
We also verify that Conjecture~\mref{conj:aim}~\mref{it:aimk} is true
for strictly decreasing partitions. Formally, we obtain the following three theorems,
inspired by the method in~\cite{BMPS19,BMS} with studying of Euler characteristics of vector bundles on the flag variety~\cite{Br,Ch,Pan,SW}.

\begin{theorem}
Let $\lambda$ be a partition of length $\ell$. Then
$
G_{1^\ell}^{\perp}g_{\lambda} = g_{\lambda-1^\ell}
$.
\mlabel{thm:aim1}
\end{theorem}

Theorem~\mref{thm:aim1} is exactly Conjecture~\mref{conj:aim}~\mref{it:aimalt} in the case of $k$ being a large enough positive integer.
Indeed, let $\lambda$ be a partition of length $\ell$ and $k\in\ZZ_{\geq 0}$ be large enough such that $\lambda_1+\ell-1\leq k$.
Then by \cite[Proposition~2.16~(d)]{BMS},
$$\fg{\lambda}{k}=g_\lambda,\quad \fg{\lambda-1^\ell}{k}=g_{\lambda-1^\ell}\, \text{ with }\, |\lambda| - |\lambda-1^\ell|-\ell = -2\ell.$$
Hence Theorem~\mref{thm:aim1} is exactly
\begin{equation*}
G_{1^\ell}^{\perp}\fg{\lambda}{k} = \fg{\lambda-1^\ell}{k} \,\text{ with }\,(-1)^{|\lambda| - |\lambda-1^\ell|-\ell}\cdot 1 = 1\in\ZZ_{\geq0}.
\end{equation*}

Denote by $$\spkl:=\{ \lambda\in\pkl\mid \lambda_x>\lambda_{x+1}, \forall x\in[\ell-1]\}$$
the set of strictly decreasing $k$-bounded partitions of length $\ell$, with the convention ${\rm sP}_0^k:={\rm P}_0^k:=\emptyset$.

\begin{theorem}
Let $\lambda\in\spkl$ and $m\in\ZZ_{\geq0}$. Then the coefficients $c_{\lambda\mu}$ in
\begin{equation*}
G_{1^m}^{\perp}\fg{\lambda}{k} = \sum_{\mu\in{\rm P}} c_{\lambda\mu} \fg{\mu}{k}
\end{equation*}
satisfy $(-1)^{|\lambda|-|\mu|-m}c_{\lambda\mu}\in\ZZ_{\geq0}$.
\mlabel{thm:aim2}
\end{theorem}

\begin{theorem}
Let $\lambda\in\spkl$. Then the coefficients $a_{\lambda\mu}$ in
\begin{equation*}
\fg{\lambda}{k} = \sum_{\mu\in{\rm P}}a_{\lambda\mu}\fg{\mu}{k+1}
\end{equation*}
satisfy $(-1)^{|\lambda|-|\mu|} a_{\lambda\mu} \in\ZZ_{\geq0}$.
\mlabel{thm:aim3}
\end{theorem}

{\bf Outline.} In Section~\mref{sec:equiv}, we first recall some background notions and results on Katalan functions and their combinatorial interpretations, especially the concept of closed $k$-Schur Katalan function as a particular case, lowering operators and the Mirror Lemma.
Then, we prove a new Mirror lemma (Lemma~\mref{lem:mirr2}) and a conclusion of Katalan functions for later use (Proposition~\mref{prop:multin}).
Section~\mref{sec:lower} is devoted to give an explicit study of the action of lowering operators $L_z$ on closed $k$-Schur Katalan functions $\fg{\lambda}{k}$ (Theorem~\mref{thm:finalform}).
In Section~\mref{sec:comproof}, we provide in detail the proofs of the above mentioned
Theorems~\mref{thm:aim1},~\mref{thm:aim2} and~\mref{thm:aim3}, respectively. The most significant result of this paper is Theorems~\mref{thm:aim2}, whose proof requires Theorem~\mref{thm:finalform} derived in Section~\mref{sec:lower}.

\vskip 0.1in

{\bf Notations.}
We fix some notations used throughout the paper.
\begin{enumerate}
\item Write $\ZZ$ to denote the set of all integers.

\item For two nonnegative integers $a\leq b$, let $[a,b]:=\{ a, a+1,\ldots,b \}$ and $[b]:=[1,b]$ with the convention $[0]:=\emptyset$.

\item Denote by $\epsilon_{i}$ the sequence of length $\ell$ with a $1$ in position $i$ and $0$'s elsewhere. Set $\varepsilon_{\alpha}:=\epsilon_{i}-\epsilon_{j}$ with $\alpha=(i,j)$.

\item
Let ${\rm P}$ be the set of partitions and ${\rm P}^{k}$ (resp. ${\rm P}_{\ell}^{k}$) be the set of $k$-bounded partitions (resp.  $k$-bounded partitions of length $\ell\in\mathbb{Z}_{\geq0}$) for $k\in \ZZ_{\geq 0}$.

\item For $\gamma := (\gamma_1,\ldots,\gamma_\ell)\in\ZZ^\ell$ with $\ell\in\ZZ_{\geq 0}$, define the {\bf size} of $\gamma$ by $|\gamma|:= \sum_{i=1}^\ell\gamma_i$.


\item For $d\in \ZZ_{\geq 0}$, denote by $e_d$ (resp. $h_d$) the elementary (resp. complete homogeneous) symmetric function.
\end{enumerate}

\section{Katalan functions and a new Mirror Lemma}
\mlabel{sec:equiv}
In this section, we first review some basic notations and facts of Katalan functions, including the concept of closed $k$-Schur Katalan function as a particular case. Then, we prove a new Mirror Lemma as a key result for later proofs of the main results in this paper.

\subsection{Katalan functions and closed $k$-Schur Katalan functions}\mlabel{ss:Katalan}
This subsection is devoted to recall the concepts of Katalan function and closed $k$-Schur Katalan function.

For $\ell\in\ZZ_{\geq 0}$, consider the set $$\Delta_{\ell}^{+}:=\{(i,j) \mid 1\le i<j\le\ell \},$$
with the convention $\Delta_{0}^{+}:=\Delta_{1}^{+} := \emptyset$.
A $\textbf{root ideal}$ $\Psi$ is an upper order ideal of the poset $\Delta_{\ell}^{+}$ with partial order given by $(a,b)\le(c,d)$ if $a\ge c$ and $b\le d$. Namely, if $(a,b)\in\Psi$, then $(c,d)\in\Psi$ for $(c,d)\in\Delta_{\ell}^{+}$ with $(c,d)\ge(a,b)$.
Let $m\in\mathbb{Z}$ and $r\in\mathbb{Z}_{\geq 0}$.
Recall $h_m$ is the complete homogeneous symmetric function, and denote by $k_{m}^{(r)}$ the
inhomogeneous version of $h_m$ given by
\begin{equation*}
    k_{m}^{(r)} :=\sum_{i=0}^{m}\binom{r+i-1}{i}h_{m-i}.
\end{equation*}
Notice that $h_{m}$ is the monomial of highest degree of $k_{m}^{(r)}$.
Here we use the convention that
$$k_{m}^{(r)}=0\,\text{ for }\,  m<0 \,\text{ and }\, k_{m}^{(0)}=h_{m}.$$
For $\gamma\in\mathbb{Z}^{\ell}$, define $$g_{\gamma}:=\text{det}(k_{\gamma_{i}+j-i}^{(i-1)})_{1\le i,j\le\ell}.$$
If $\gamma$ is a partition, $g_{\gamma}$ is called the {\bf dual stable Grothendieck polynomial} studied in~\cite{Las01,Le}.
The following is the concept of Katalan function.

\begin{defn}(\cite[Definition~2.1]{BMS})
For a root ideal $\Psi\subseteq\Delta_{\ell}^{+}$, a multiset $M$ with $\text{supp}(M)\subseteq [\ell]$ and $\gamma\in\mathbb{Z}^{\ell}$, define the $\textbf{Katalan function}$
\begin{equation*}
K(\Psi;M;\gamma):=\prod_{z\in M}(1-L_{z})\prod_{(i,j)\in\Psi}(1-R_{ij})^{-1}g_{\gamma},
\end{equation*}
where $L_{z}$ is a $\textbf{lowering operator}$ acting on $g_{\gamma}$ by $L_{z}g_{\gamma}:=g_{\gamma-\epsilon_{z}}$ and $R_{ij}$ is an operator acting on $g_{\gamma}$ by $R_{ij}g_{\gamma}:=g_{\gamma+\epsilon_{i}-\epsilon_{j}}$.
\mlabel{defn:KataFunc}
\end{defn}

We collect here some simple but useful facts.

\begin{remark}
Let $\Psi\subseteq\Delta_{\ell}^{+}$ be a root ideal, $M$ a multiset with $\text{supp}(M)\subseteq [\ell]$ and $\gamma\in\mathbb{Z}^{\ell}$.
\begin{enumerate}
\item By~\cite[p.18]{BMS},
$$L_z K(\Psi;M;\gamma)=K(\Psi;M;\gamma-\epsilon_z), \quad \forall z\in[\ell].$$

\item From~\cite[Proposition~2.3~(b)]{BMS}, $$K(\varnothing;\varnothing;\gamma) = g_\gamma.$$

\item
In terms of~\cite[(2.3)]{BMS},
\begin{equation*}
g_{\gamma}=\prod_{1\le i<j\le\ell}(1-R_{ij})k_{\gamma}, \,\text{ where } \, k_{\gamma}:= k_{\gamma_{1}}^{(0)} \cdots k_{\gamma_{\ell}}^{(\ell-1)}.
\end{equation*}
So a Katalan function has another expression
\begin{equation*}
K(\Psi;M;\gamma)=\prod_{z\in M}(1-L_{z})\prod_{(i,j)\in\Delta_{\ell}^{+}\setminus\Psi}(1-R_{ij})k_{\gamma},
\end{equation*}
where $L_{z}k_{\gamma}=k_{\gamma-\epsilon_{z}}$ and $R_{ij}k_{\gamma}=k_{\gamma+\epsilon_{i}-\epsilon_{j}}$. Expanding the right hand side of the above equation, $K(\Psi;M;\gamma)$ is a linear summation of some $k_\mu$'s with $\mu\in \ZZ^{\ell}$.

\end{enumerate}
\mlabel{re:BMS}
\end{remark}

Given a multiset $M$ on $[\ell]$ and $a\in M$, denote by $m_{M}(a)$ the number of occurrences of $a$ in $M$. A Katalan function $K(\Psi;M;\gamma)$ can be represented by the $\ell\times\ell$ grid of boxes (labelled by matrix-style coordinates), with the boxes of $\Psi$ shaded, $m_{M}(a)$ $\bullet$'s in column $a$ and the entries of $\gamma$ on the diagonal of these boxes. Let us expose an example for illustration.

\begin{exam}
For $\ell=4$, $\Psi=\{ (1,3),(1,4),(2,4) \}$, $M=\{2,3,4,4\}$ and $\gamma=(3,2,1,3)$, we represent the related Katalan function by
\begin{equation*}
K(\Psi;M;\gamma)=
\begin{tikzpicture}[scale=.4,line width=0.5pt,baseline=(a.base)]
\draw (0,2) rectangle (1,1);\node at(0.5,1.5){\scriptsize\( 3 \)};
\draw (1,2) rectangle (2,1);\node at(1.5,1.5){\scriptsize\( \bullet \)};
\filldraw[red,draw=black] (2,2) rectangle (3,1);\node at(2.5,1.5){\scriptsize\( \bullet \)};
\filldraw[red,draw=black] (3,2) rectangle (4,1);\node at(3.5,1.5){\scriptsize\( \bullet \)};
\draw (0,1) rectangle (1,0);
\draw (1,1) rectangle (2,0);\node at(1.5,0.5){\scriptsize\( 2 \)};
\draw (2,1) rectangle (3,0);
\filldraw[red,draw=black] (3,1) rectangle (4,0);\node at(3.5,0.5){\scriptsize\( \bullet \)};
\node (a) [align=center] {\\[-5pt] };
\draw (0,0) rectangle (1,-1);
\draw (1,0) rectangle (2,-1);
\draw (2,0) rectangle (3,-1);\node at(2.5,-0.5){\scriptsize\( 1 \)};
\draw (3,0) rectangle (4,-1);
\draw (0,-1) rectangle (1,-2);
\draw (1,-1) rectangle (2,-2);
\draw (2,-1) rectangle (3,-2);
\draw (3,-1) rectangle (4,-2);\node at(3.5,-1.5){\scriptsize\( 3 \)};
\end{tikzpicture}.
\end{equation*}
\end{exam}

For $\lambda\in\text{P}_{\ell}^{k}$, define a root ideal
\begin{equation}
	\Delta^{k}(\lambda):=\{(i,j)\in\Delta_{\ell}^{+} \mid k-\lambda_{i}+i<j \} \subseteq \Delta_{\ell}^{+}.
	\mlabel{eq:dkl}
\end{equation}
For $\mathcal{R}\subseteq\Delta_{\ell}^{+}$, denote by
\begin{equation}
L(\mathcal{R}):=\bigsqcup_{(i,j)\in\mathcal{R}}\{j\}
\mlabel{eq:Lroot}
\end{equation}
the multiset of second components of elements in $\mathcal{R}$.
In particular, if $\mathcal{R}$ is a root ideal, we denote $K(\Psi;\mathcal{R};\gamma):= K(\Psi;L(\mathcal{R});\gamma)$.
Now we come to the main concept studied in the present paper.

\begin{defn}(\cite[Definition 2.11]{BMS})
For $\lambda\in\text{P}_{\ell}^{k}$, the {\bf closed $k$-Schur Katalan function} is
\begin{equation*}
	\tilde{\mathfrak{g}}_{\lambda}^{(k)} :=K(\Delta^{k}(\lambda);\Delta^{k}(\lambda);\lambda).
\end{equation*}
\mlabel{defn:cKataFunc}
\end{defn}

\subsection{A new Mirror Lemma}
This subsection is devoted to prove a new Mirror Lemma for later uses.
For this, let us review some necessary concepts and the Mirror Lemma of Katalan functions~\cite{BMPS19,BMS}.

Let $\Psi\subseteq\Delta_{\ell}^{+}$ be a root ideal and $x\in[\ell]$.
\begin{enumerate}
\item A root $\alpha\in\Psi$ is called {\bf removable} if $\Psi\setminus \alpha$ is still a root ideal; a root $\alpha\in\Delta_{\ell}^{+}\setminus\Psi$ is called {\bf addable} if $\Psi\cup\alpha$ is also a root ideal.

\item If there exists $j\in[\ell]$ such that $(x,j)$ is removable in $\Psi$, we denote $\text{down}_{\Psi}(x):=j$, else we say that $\text{down}_{\Psi}(x)$ is undefined; if there exists $j\in[\ell]$ such that $(j,x)$ is removable in $\Psi$, we denote $\text{up}_{\Psi}(x):=j$, else we say that $\text{up}_{\Psi}(x)$ is undefined.

\item The {\bf bounce graph} of $\Psi$ is a graph on the vertex set $[\ell]$ with edges $(x,\text{down}_{\Psi}(x))$ for each $x\in[\ell]$ such that $\text{down}_{\Psi}(x)$ is defined. The bounce
    graph of $\Psi$ is a disjoint union of paths called {\bf bounce paths} of $\Psi$.

\item Denote by $\text{top}_{\Psi}(x)$ (resp. $\text{bot}_{\Psi}(x)$) the minimum (resp. maximum) vertex, as a number in $[\ell]$, in the unique bounce path of $\Psi$ containing $x$.

\item For $a,b\in[\ell]$ in the same bounce path of $\Psi$ with $a\le b$, define
\begin{align*}
\text{path}_{\Psi}(a,b):=(a,\text{down}_{\Psi}(a),
\text{down}_{\Psi}^{2}(a),\ldots,b),
\end{align*}
and $|\text{path}_{\Psi}(a,b)|$ to be the length of $\text{path}_{\Psi}(a,b)$.
Denote $\text{uppath}_{\Psi}(x):=\text{path}_{\Psi}(\text{top}_{\Psi}(x),x)$ for $x\in[\ell]$.
\end{enumerate}

\begin{defn}
A root ideal $\Psi\subseteq\Delta_{\ell}^{+}$ is said to have
\begin{enumerate}
\item a {\bf wall} in rows $r, r+1$ if rows $r, r+1$ have the same length;

\item a {\bf ceiling} in columns $c, c+1$ if columns $c, c+1$ have the same length;

\item a {\bf mirror} in rows $r, r+1$ if $\Psi$ has removable roots $(r,c)$ and $(r+1,c+1)$ for some $c\in[r+2,\ell-1]$;

\item the {\bf bottom} in row $\bott$ of $\Psi$ if $(\bott,p)\in\Psi$ for some $p\in [\ell]$ and $(\bott+1,q)\notin\Psi$ for all $q\in [\ell]$.
\end{enumerate}
\end{defn}

\begin{lemma}(\cite[Lemma 3.3]{BMS})
Let $\Psi\subseteq\Delta_{\ell}^{+}$ be a root ideal, $M$ a multiset with \textup{supp}$(M)\subseteq [\ell]$ and $\gamma\in\mathbb{Z}^{\ell}$. Let $z\in[\ell]$ such that
\begin{enumerate}
\item $\Psi$ has a ceiling in columns $z,z+1$;

\item $\Psi$ has a wall in rows $z,z+1$;

\item $m_M(z)+1 = m_M(z+1)$.
\end{enumerate}
Then
\begin{equation*}
K(\Psi;M;\gamma) + K(\Psi;M;s_z\gamma-\epsilon_z+\epsilon_{z+1}) = 0,
\end{equation*}
where $s_z$ is a simple permutation in the symmetric group $S_\ell$ of degree $\ell$ with
$$s_z\gamma := (\gamma_1,\ldots,\gamma_{z-1},\gamma_{z+1},\gamma_z,\gamma_{z+2},\ldots,
\gamma_\ell).$$
\mlabel{lem:K+K=0}
\end{lemma}

\begin{lemma}(\cite[Proposition 3.9]{BMS})
Let $\Psi\subseteq\Delta_{\ell}^{+}$ be a root ideal, $M$ a multiset with \textup{supp}$(M)\subseteq [\ell]$ and $\gamma\in\mathbb{Z}^{\ell}$. Then
\begin{enumerate}
\item for any removable root $\beta$ of $\Psi$,
\begin{equation*}
K(\Psi;M;\gamma)=K(\Psi\setminus\beta;M;\gamma)+K(\Psi;M;\gamma+\varepsilon_{\beta});
\end{equation*}
\mlabel{it:relk1}

\item for any addable root $\alpha$ of $\Psi$,
\begin{equation*}
K(\Psi;M;\gamma)=K(\Psi\cup\alpha;M;\gamma)
-K(\Psi\cup\alpha;M;\gamma+\varepsilon_{\alpha});
\end{equation*}
\mlabel{it:relk2}

\item for any $m\in M$,
\begin{equation*}
K(\Psi;M;\gamma)=K(\Psi;M\setminus m;\gamma)-K(\Psi;M\setminus m;\gamma-\epsilon_{m}),
\end{equation*}
where $M\setminus m$ means removing only one element $m$ from $M$;
\mlabel{it:relk3}

\item for any $m\in[\ell]$,
\begin{equation*}
K(\Psi;M;\gamma)=K(\Psi;M\sqcup m;\gamma)+K(\Psi;M;\gamma-\epsilon_{m}).
\end{equation*}
\mlabel{it:relk4}
\end{enumerate}
\mlabel{lem:relk}
\end{lemma}

\begin{lemma}({\rm Mirror Lemma}, \cite[Lemma 4.6]{BMS})
Let $\Psi\subseteq\Delta_{\ell}^{+}$ be a root ideal, $M$ a multiset with \textup{supp}$(M)\subseteq [\ell]$ and $\gamma\in\mathbb{Z}^{\ell}$. Let $1\le y\le z<\ell$ be indices in a same bounce path  of $\Psi$ satisfying
\begin{enumerate}
\item $\Psi$ has a ceiling in columns $y,y+1$;
\item $\Psi$ has a mirror in rows $x,x+1$ for each $x\in\textup{path}_{\Psi}(y,\textup{up}_{\Psi}(z))$;
\item $\Psi$ has a wall in rows $z,z+1$;
\item $\gamma_{x}=\gamma_{x+1}$ for all $x\in\textup{path}_{\Psi}(y,\textup{up}_{\Psi}(z))$;
\item $\gamma_{z}+1=\gamma_{z+1}$;
\item $m_{M}(x)+1=m_{M}(x+1)$ for all $x\in\textup{path}_{\Psi}(\textup{down}_{\Psi}(y),z)$.
\end{enumerate}
If $m_{M}(y)+1=m_{M}(y+1)$, then $K(\Psi;M;\gamma)=0$; if $m_{M}(y)=m_{M}(y+1)$, then $K(\Psi;M;\gamma)=K(\Psi;M;\gamma-\epsilon_{z+1})$.
\mlabel{lem:mirr1}
\end{lemma}

The following is a new Mirror Lemma, which is a key result serving for our later uses.

\begin{lemma}({\rm Mirror Lemma II})
Let $\Psi\subseteq\Delta_{\ell}^{+}$ be a root ideal, $M$ a multiset with \textup{supp}$(M)\subseteq [\ell]$ and $\gamma\in\mathbb{Z}^{\ell}$. Let $1\le y\le z<\ell$ be indices in a same bounce path  of $\Psi$. Suppose that
\begin{enumerate}
\item $\Psi$ has an addable root $\alpha=(\textup{up}_{\Psi}(y+1),y)$ and a removable root $\beta=(\textup{up}_{\Psi}(y+1),y+1)$;
\item $\Psi$ has a mirror in rows $x,x+1$ for each $x\in\textup{path}_{\Psi}(y,\textup{up}_{\Psi}(z))$;
\item $\Psi$ has a wall in rows $z,z+1$;
\item $m_{M}(x)+1=m_{M}(x+1)$ for all $x\in\textup{path}_{\Psi}(y,z)$;
\item $\gamma_{x}=\gamma_{x+1}$ for all $x\in\textup{path}_{\Psi}(y,\textup{up}_{\Psi}(z))$;
\item $\gamma_{z}+1=\gamma_{z+1}$.
    \end{enumerate}
Then
\begin{equation*}
K(\Psi;M;\gamma)=
K(\Psi\cup\alpha;M\sqcup y;\gamma+\epsilon_{\text{\rm up}_{\Psi}(y+1)}
-\epsilon_{z+1}).
\end{equation*}
\mlabel{lem:mirr2}
\end{lemma}

\begin{proof}
We apply induction on $z-y\geq 0$. For the initial step of $z-y=0$,
\begin{align}
K(\Psi;M;\gamma)=&\ K(\Psi\cup\alpha;M;\gamma) - K(\Psi\cup\alpha;M;\gamma + \varepsilon_{\alpha}) \hspace{1cm} (\text{by Lemma~\ref{lem:relk}~\ref{it:relk2}})\notag\\
=&\  - K(\Psi\cup\alpha;M;\gamma + \varepsilon_{\alpha}) \hspace{1cm} (\text{the first summand vanishes by Lemma~\ref{lem:mirr1}})\notag\\
=&\  - K(\Psi\cup\alpha;M;\gamma + \epsilon_{{\rm up}_\Psi(z+1)}-\epsilon_z) \hspace{1cm} (\text{by $\alpha=(\textup{up}_{\Psi}(y+1),y)$ and $z=y$})\notag\\
=&\ K(\Psi\cup\alpha;M;\gamma+\epsilon_{\text{\rm up}_{\Psi}(z+1)}
-\epsilon_{z+1})\notag\\
& \hspace{0.5cm} (\text{by $s_z(\gamma + \epsilon_{{\rm up}_\Psi(z+1)}-\epsilon_z)-\epsilon_z +\epsilon_{z+1} = \gamma+\epsilon_{\text{\rm up}_{\Psi}(z+1)}
-\epsilon_{z+1}$ and Lemma~\ref{lem:K+K=0}})\notag\\
=&\ K(\Psi\cup\alpha;M\sqcup z;\gamma+\epsilon_{\text{\rm up}_{\Psi}(z+1)}
-\epsilon_{z+1}) + K(\Psi\cup\alpha;M;\gamma+\epsilon_{\text{\rm up}_{\Psi}(z+1)}
-\epsilon_{z+1}-\epsilon_{z}) \mlabel{eq:mirr2-ini} \\
&\ \hspace{9cm} (\text{by Lemma~\ref{lem:relk}~\ref{it:relk4}}).\notag
\end{align}
Since $\Psi\cup\alpha$ has a ceiling in columns $z,z+1$, a wall in rows $z,z+1$ and
\[
\big(\gamma+\epsilon_{\text{\rm up}_{\Psi}(z+1)}
-\epsilon_{z+1}-\epsilon_{z}\big)_{z} + 1 = \gamma_z-1+1 \overset{\text{(f)}}{=} (\gamma_{z+1}-1)-1+1  = \big(\gamma+\epsilon_{\text{\rm up}_{\Psi}(z+1)}
-\epsilon_{z+1}-\epsilon_{z}\big)_{z+1}
\]
by Lemma~\mref{lem:mirr1}, we obtain
\[
K(\Psi\cup\alpha;M;\gamma+\epsilon_{\text{\rm up}_{\Psi}(z+1)}
-\epsilon_{z+1}-\epsilon_{z}) = 0.
\]
Substituting the above equation into~(\mref{eq:mirr2-ini}) and using $z=y$ yield
\begin{align*}
K(\Psi;M;\gamma)=
K(\Psi\cup\alpha;M\sqcup z;\gamma+\epsilon_{\text{\rm up}_{\Psi}(z+1)}
-\epsilon_{z+1}) =
K(\Psi\cup\alpha;M\sqcup y;\gamma+\epsilon_{\text{\rm up}_{\Psi}(y+1)}
-\epsilon_{z+1}).
\end{align*}

For the inductive step of $z-y>0$, denote $\alpha_1:=({\rm up}_\Psi(z+1),z)$.
Since there is a mirror in rows ${\rm up}_\Psi(z),{\rm up}_\Psi(z+1)$, we have that $\alpha_1$ is an addable root of $\Psi$. Hence
\begin{align}
K(\Psi;M;\gamma)=&\ K(\Psi\cup\alpha_1;M;\gamma) - K(\Psi\cup\alpha_1;M;\gamma + \varepsilon_{\alpha_1}) \hspace{1cm} (\text{by Lemma~\ref{lem:relk}~\ref{it:relk2}})\notag\\
=&\  - K(\Psi\cup\alpha_1;M;\gamma + \varepsilon_{\alpha_1}) \hspace{1cm} (\text{the first summand vanishes by Lemma~\ref{lem:mirr1}})\notag\\
=&\  - K(\Psi\cup\alpha_1;M;\gamma + \epsilon_{{\rm up}_\Psi(z+1)}-\epsilon_z)\hspace{1cm}(\text{by $\alpha_1:=({\rm up}_\Psi(z+1),z)$})\notag\\
=&\ K(\Psi\cup\alpha_1;M;\gamma + \epsilon_{{\rm up}_\Psi(z+1)}-\epsilon_{z+1})\mlabel{eq:mirr2-2-1}\\
&\ (\text{by $s_z(\gamma + \epsilon_{{\rm up}_\Psi(z+1)}-\epsilon_z)-\epsilon_z +\epsilon_{z+1} = \gamma+\epsilon_{\text{\rm up}_{\Psi}(z+1)}
-\epsilon_{z+1}$ and Lemma~\ref{lem:K+K=0}}).\notag
\end{align}
Since $\Psi\cup\alpha_1$ has a wall in rows ${\rm up}_\Psi(z),{\rm up}_\Psi(z+1)$ and
\begin{align*}
&\ \big( \gamma + \epsilon_{{\rm up}_\Psi(z+1)}-\epsilon_{z+1} \big)_{{\rm up}_\Psi(z)} + 1 = \gamma_{{\rm up}_\Psi(z)} + 1 \\ \overset{\text{(e)}}{=}&\ \gamma_{{\rm up}_\Psi(z)+1} + 1
= \gamma_{{\rm up}_\Psi(z+1)} + 1 = \big( \gamma + \epsilon_{{\rm up}_\Psi(z+1)}-\epsilon_{z+1} \big)_{{\rm up}_\Psi(z+1)},
\end{align*}
we have
\begin{align}
K(\Psi;M;\gamma) =&\ K(\Psi\cup\alpha_1;M;\gamma + \epsilon_{{\rm up}_\Psi(z+1)}-\epsilon_{z+1})\hspace{1cm}(\text{by~(\ref{eq:mirr2-2-1})})\notag\\
=&\ K(\Psi\cup\{\alpha_1,\alpha\};M\sqcup y;\gamma + \epsilon_{{\rm up}_\Psi(z+1)}-\epsilon_{z+1} + \epsilon_{{\rm up}_\Psi(y+1)}-\epsilon_{{\rm up}_\Psi(z+1)})\notag\\
&\ \hspace{5cm}(\text{by the inductive hypothesis})\notag\\
=&\ K(\Psi\cup\{\alpha_1,\alpha\};M\sqcup y;\gamma+ \epsilon_{{\rm up}_\Psi(y+1)}-\epsilon_{z+1} )\notag\\
=&\ K(\Psi\cup\alpha;M\sqcup y;\gamma+ \epsilon_{{\rm up}_\Psi(y+1)}-\epsilon_{z+1} )\mlabel{eq:mirr2-ind}\\
&\ + K(\Psi\cup\{\alpha_1,\alpha\};M\sqcup y;\gamma+ \epsilon_{{\rm up}_\Psi(y+1)}-\epsilon_{z+1} + \epsilon_{{\rm up}_\Psi(z+1)}-\epsilon_{z}) \notag\\
&\ \hspace{4cm} (\text{by removing $\alpha_1$ and Lemma~\ref{lem:relk}~\ref{it:relk1}}).\notag
\end{align}
Since $\Psi\cup\{\alpha_1,\alpha\}$ has a ceiling in columns $y,y+1$, a wall in rows ${\rm up}_\Psi(z),{\rm up}_\Psi(z+1)$ and
\begin{align*}
\big( \gamma+ \epsilon_{{\rm up}_\Psi(y+1)}-\epsilon_{z+1} + \epsilon_{{\rm up}_\Psi(z+1)}-\epsilon_{z} \big)_{{\rm up}_\Psi(z)} + 1 =&\ \gamma_{{\rm up}_\Psi(z)} + 1\\
=&\ \gamma_{{\rm up}_\Psi(z)+1} + 1\hspace{1cm}(\text{by (e)})\\
=&\ \gamma_{{\rm up}_\Psi(z+1)} + 1\hspace{1cm}(\text{by ${\rm up}_\Psi(z)+1={\rm up}_\Psi(z+1)$})\\
=&\ \big( \gamma+ \epsilon_{{\rm up}_\Psi(y+1)}-\epsilon_{z+1} + \epsilon_{{\rm up}_\Psi(z+1)}-\epsilon_{z} \big)_{{\rm up}_\Psi(z+1)},
\end{align*}
it follows from Lemma~\mref{lem:mirr1} that
\[
K(\Psi\cup\{\alpha_1,\alpha\};M\sqcup y;\gamma+ \epsilon_{{\rm up}_\Psi(y+1)}-\epsilon_{z+1} + \epsilon_{{\rm up}_\Psi(z+1)}-\epsilon_{z}) = 0,
\]
which, together with~(\ref{eq:mirr2-ind}), implies that
\[
K(\Psi;M;\gamma)= K(\Psi\cup\alpha;M\sqcup y;\gamma+ \epsilon_{{\rm up}_\Psi(y+1)}-\epsilon_{z+1} ).
\]
This completes the proof.
\end{proof}

We end this subsection with the following property of Katalan functions.

\begin{prop}
Let $\Psi\subseteq\Delta_{\ell}^{+}$ be a root ideal, $M$ a multiset with \textup{supp}$(M)\subseteq [\ell]$, $\gamma\in\mathbb{Z}^{\ell}$ and $z\in[\ell]$. If $n>\gamma_{z}+\ell-z$, then
$L_{z}^{n} K(\Psi;M;\gamma)=0.
$
\mlabel{prop:multin}
\end{prop}

\begin{proof}
Notice that $k_{m}^{(r)}=0$ for $m<0$.
In terms of Remark~\mref{re:BMS}~(c), a Katalan function $K(\Psi;M;\gamma)$ can be written as a linear summation
of some $k_{\mu}$'s with $\mu\in\mathbb{Z}^{\ell}$.
For each summand $k_{\mu}$,
\[
n>\gamma_{z}+\ell-z\geq \mu_z\,\text{ and so }\, \mu_z-n<0,
\]
whence
$$L_z^n k_{\mu} = k_{\mu - n\epsilon_z} = k_{\mu_{1}}^{(0)} \cdots  k_{\mu_{z}-n}^{(z-1)}\cdots k_{\mu_{\ell}}^{(\ell-1)}=0.$$
Hence the result is valid by the additivity of the lowering operator $L_z$.
\end{proof}

\section{Lowering operators acting on closed $k$-Schur Katalan functions}
\mlabel{sec:lower}
This section is to exploit on the effect of lowering operators $L_z$ acting on closed $k$-Schur Katalan functions $\tilde{\mathfrak{g}}_{\lambda}^{(k)}$. The main results are summarized in Theorem~\mref{thm:finalform}. Let us expose our idea first.
For $\lambda\in\pkl$, it follow from Definition~\ref{defn:cKataFunc} and Remark~\ref{re:BMS}~(a) that
\begin{equation}
L_z \fg{\lambda}{k} =L_z K(\dkl;\dkl;\lambda)
=K(\dkl;\dkl;\lambda-\epsilon_z).
\mlabel{eq:LamToMu-a}
\end{equation}
Here since $\lambda - \epsilon_z$ may not be a partition,
we generalize closed $k$-Schur Katalan functions $$K(\Delta^{k}(\lambda);\Delta^{k}(\lambda);\lambda)$$
by relaxing the condition that $\lambda$ is a partition.
For this, recall~\mcite{BMPS19}
$$\tilde{\text{P}}_{\ell}^{k}:=\{ \mu\in\mathbb{Z}_{\le k}^{\ell} \mid \mu_{1}+\ell-1\ge \mu_{2}+\ell-2\ge\cdots\ge\mu_{\ell} \},$$
which is employed to study Catalan functions and $k$-schur positivity.
Notice that $\pkl\subseteq\tilde{\text{P}}_{\ell}^{k}$.

\begin{remark}
Let $\mu\in\tilde{\text{P}}_{\ell}^{k}$.
\begin{enumerate}
\item The $\Delta^{k}(\mu):=\{(i,j)\in\Delta_{\ell}^{+} \mid k-\mu_{i}+i<j \}$ is a root ideal~\cite[p.~943]{BMPS19}.

\item In the root ideal $\Delta^{k}(\mu)$, for $z\in[\ell-1]$,
\begin{itemize}
\item if $\mu_{z}+1=\mu_{z+1}$, then there is a wall in rows $z,z+1$;

\item if $k>\mu_{z}=\mu_{z+1}$ and $z\in[\bott-1]$ with $\bott:=\bott_\mu$ the bottom of $\dkmu$, then there is a mirror in rows $z,z+1$;

\item if $\mu_{z}>\mu_{z+1}$ and $\down{\dkmu}{z}+1\leq\ell$, then there is a ceiling in columns
    $\textup{down}_{\Delta^{k}(\mu)}(z)$, $\textup{down}_{\Delta^{k}(\mu)}(z)+1$~\cite[Remark 5.1]{BMS}.
\end{itemize}

\item Let $z\in[\bott]$ with $\bott:=\bott_\mu$ the bottom of $\dkmu$. If $\mu_z\geq\mu_{z+1}$, then $(z, \down{\dkmu}{z})$ is removable in $\Delta^{k}(\mu)$.

\item Suppose that there exists $z\in[\ell-1]$ such that $\mu_{z}+1=\mu_{z+1}$ and $\mu_x\geq\mu_{x+1}$ for all $x\in[\ell-1]\setminus z$. For the relation between $\textup{top}_{\Delta^{k}(\mu)}(z)$ and $\textup{top}_{\Delta^{k}(\mu)}(z+1)$, there are only two possibilities:
\begin{equation}
\begin{aligned}
&{\rm (I)} \quad y_1:={\rm top}_{\dkmu}(z+1)> {\rm top}_{\dkmu}(z),\\
&{\rm (II)} \quad y_2:=\textup{top}_{\Delta^{k}(\mu)}(z)>\textup{top}_{\Delta^{k}(\mu)}(z+1).
\end{aligned}
\mlabel{eq:noteq}
\end{equation}
That is, ${\rm top}_{\dkmu}(z+1)\neq {\rm top}_{\dkmu}(z)$.
In both cases, $\Delta^{k}(\mu)$ has a mirror in rows $x,x+1$ for each $x\in\textup{path}_{\Delta^{k}(\mu)}            (y_i,\textup{up}_{\Delta^{k}(\mu)}(z))$, $i=1,2$.  See Fig.~1 (resp. Fig.~2) for an illustration of (I) (resp. (II)).
\begin{equation*}
\begin{tikzpicture}[scale=.25,line width=0.5pt,baseline=(a.base)]
\draw (0,0) rectangle (1,-1);
\draw (1,0) rectangle (2,-1);
\filldraw[red,draw=black] (2,0) rectangle (3,-1);\node at(2.5,-0.5){\tiny\( \bullet \)};
\filldraw[red,draw=black] (3,0) rectangle (4,-1);\node at(3.5,-0.5){\tiny\( \bullet \)};
\filldraw[red,draw=black] (4,0) rectangle (5,-1);\node at(4.5,-0.5){\tiny\( \bullet \)};
\filldraw[red,draw=black] (5,0) rectangle (6,-1);\node at(5.5,-0.5){\tiny\( \bullet \)};
\filldraw[red,draw=black] (6,0) rectangle (7,-1);\node at(6.5,-0.5){\tiny\( \bullet \)};
\filldraw[red,draw=black] (7,0) rectangle (8,-1);\node at(7.5,-0.5){\tiny\( \bullet \)};
\filldraw[red,draw=black] (8,0) rectangle (9,-1);\node at(8.5,-0.5){\tiny\( \bullet \)};
\filldraw[red,draw=black] (9,0) rectangle (10,-1);\node at(9.5,-0.5){\tiny\( \bullet \)};
\filldraw[red,draw=black] (10,0) rectangle (11,-1);\node at(10.5,-0.5){\tiny\( \bullet \)};
\filldraw[red,draw=black] (11,0) rectangle (12,-1);\node at(11.5,-0.5){\tiny\( \bullet \)};
\filldraw[red,draw=black] (12,0) rectangle (13,-1);\node at(12.5,-0.5){\tiny\( \bullet \)};
\filldraw[red,draw=black] (13,0) rectangle (14,-1);\node at(13.5,-0.5){\tiny\( \bullet \)};
\filldraw[red,draw=black] (14,0) rectangle (15,-1);\node at(14.5,-0.5){\tiny\( \bullet \)};
\draw (0,-1) rectangle (1,-2);
\draw (1,-1) rectangle (2,-2);
\draw (2,-1) rectangle (3,-2);
\filldraw[red,draw=black] (3,-1) rectangle (4,-2);\node at(3.5,-1.5){\tiny\( \bullet \)};
\filldraw[red,draw=black] (4,-1) rectangle (5,-2);\node at(4.5,-1.5){\tiny\( \bullet \)};
\filldraw[red,draw=black] (5,-1) rectangle (6,-2);\node at(5.5,-1.5){\tiny\( \bullet \)};
\filldraw[red,draw=black] (6,-1) rectangle (7,-2);\node at(6.5,-1.5){\tiny\( \bullet \)};
\filldraw[red,draw=black] (7,-1) rectangle (8,-2);\node at(7.5,-1.5){\tiny\( \bullet \)};
\filldraw[red,draw=black] (8,-1) rectangle (9,-2);\node at(8.5,-1.5){\tiny\( \bullet \)};
\filldraw[red,draw=black] (9,-1) rectangle (10,-2);\node at(9.5,-1.5){\tiny\( \bullet \)};
\filldraw[red,draw=black] (10,-1) rectangle (11,-2);\node at(10.5,-1.5){\tiny\( \bullet \)};
\filldraw[red,draw=black] (11,-1) rectangle (12,-2);\node at(11.5,-1.5){\tiny\( \bullet \)};
\filldraw[red,draw=black] (12,-1) rectangle (13,-2);\node at(12.5,-1.5){\tiny\( \bullet \)};
\filldraw[red,draw=black] (13,-1) rectangle (14,-2);\node at(13.5,-1.5){\tiny\( \bullet \)};
\filldraw[red,draw=black] (14,-1) rectangle (15,-2);\node at(14.5,-1.5){\tiny\( \bullet \)};
\draw (0,-2) rectangle (1,-3);
\draw (1,-2) rectangle (2,-3);
\draw (2,-2) rectangle (3,-3);
\draw (3,-2) rectangle (4,-3);
\filldraw[red,draw=black] (4,-2) rectangle (5,-3);\node at(4.5,-2.5){\tiny\( \bullet \)};
\filldraw[red,draw=black] (5,-2) rectangle (6,-3);\node at(5.5,-2.5){\tiny\( \bullet \)};
\filldraw[red,draw=black] (6,-2) rectangle (7,-3);\node at(6.5,-2.5){\tiny\( \bullet \)};
\filldraw[red,draw=black] (7,-2) rectangle (8,-3);\node at(7.5,-2.5){\tiny\( \bullet \)};
\filldraw[red,draw=black] (8,-2) rectangle (9,-3);\node at(8.5,-2.5){\tiny\( \bullet \)};
\filldraw[red,draw=black] (9,-2) rectangle (10,-3);\node at(9.5,-2.5){\tiny\( \bullet \)};
\filldraw[red,draw=black] (10,-2) rectangle (11,-3);\node at(10.5,-2.5){\tiny\( \bullet \)};
\filldraw[red,draw=black] (11,-2) rectangle (12,-3);\node at(11.5,-2.5){\tiny\( \bullet \)};
\filldraw[red,draw=black] (12,-2) rectangle (13,-3);\node at(12.5,-2.5){\tiny\( \bullet \)};
\filldraw[red,draw=black] (13,-2) rectangle (14,-3);\node at(13.5,-2.5){\tiny\( \bullet \)};
\filldraw[red,draw=black] (14,-2) rectangle (15,-3);\node at(14.5,-2.5){\tiny\( \bullet \)};
\draw (0,-3) rectangle (1,-4);
\draw (1,-3) rectangle (2,-4);
\draw (2,-3) rectangle (3,-4);
\draw (3,-3) rectangle (4,-4);
\draw (4,-3) rectangle (5,-4);
\filldraw[blue!60,draw=black] (5,-3) rectangle (6,-4);\node at(5.5,-3.5){\tiny\( \bullet \)};
\filldraw[blue!60,draw=black] (6,-3) rectangle (7,-4);\node at(6.5,-3.5){\tiny\( \bullet \)};
\filldraw[red,draw=black] (7,-3) rectangle (8,-4);\node at(7.5,-3.5){\tiny\( \bullet \)};
\filldraw[red,draw=black] (8,-3) rectangle (9,-4);\node at(8.5,-3.5){\tiny\( \bullet \)};
\filldraw[red,draw=black] (9,-3) rectangle (10,-4);\node at(9.5,-3.5){\tiny\( \bullet \)};
\filldraw[red,draw=black] (10,-3) rectangle (11,-4);\node at(10.5,-3.5){\tiny\( \bullet \)};
\filldraw[red,draw=black] (11,-3) rectangle (12,-4);\node at(11.5,-3.5){\tiny\( \bullet \)};
\filldraw[red,draw=black] (12,-3) rectangle (13,-4);\node at(12.5,-3.5){\tiny\( \bullet \)};
\filldraw[red,draw=black] (13,-3) rectangle (14,-4);\node at(13.5,-3.5){\tiny\( \bullet \)};
\filldraw[red,draw=black] (14,-3) rectangle (15,-4);\node at(14.5,-3.5){\tiny\( \bullet \)};
\draw (0,-4) rectangle (1,-5);
\draw (1,-4) rectangle (2,-5);
\draw (2,-4) rectangle (3,-5);
\draw (3,-4) rectangle (4,-5);
\draw (4,-4) rectangle (5,-5);
\draw (5,-4) rectangle (6,-5);
\draw (6,-4) rectangle (7,-5);
\filldraw[red,draw=black] (7,-4) rectangle (8,-5);\node at(7.5,-4.5){\tiny\( \bullet \)};
\filldraw[red,draw=black] (8,-4) rectangle (9,-5);\node at(8.5,-4.5){\tiny\( \bullet \)};
\filldraw[red,draw=black] (9,-4) rectangle (10,-5);\node at(9.5,-4.5){\tiny\( \bullet \)};
\filldraw[red,draw=black] (10,-4) rectangle (11,-5);\node at(10.5,-4.5){\tiny\( \bullet \)};
\filldraw[red,draw=black] (11,-4) rectangle (12,-5);\node at(11.5,-4.5){\tiny\( \bullet \)};
\filldraw[red,draw=black] (12,-4) rectangle (13,-5);\node at(12.5,-4.5){\tiny\( \bullet \)};
\filldraw[red,draw=black] (13,-4) rectangle (14,-5);\node at(13.5,-4.5){\tiny\( \bullet \)};
\filldraw[red,draw=black] (14,-4) rectangle (15,-5);\node at(14.5,-4.5){\tiny\( \bullet \)};
\draw (0,-5) rectangle (1,-6);
\draw (1,-5) rectangle (2,-6);
\draw (2,-5) rectangle (3,-6);
\draw (3,-5) rectangle (4,-6);
\draw (4,-5) rectangle (5,-6);
\draw (5,-5) rectangle (6,-6);
\draw (6,-5) rectangle (7,-6);
\draw (7,-5) rectangle (8,-6);
\draw (8,-5) rectangle (9,-6);
\filldraw[blue!60,draw=black] (9,-5) rectangle (10,-6);\node at(9.5,-5.5){\tiny\( \bullet \)};
\filldraw[red,draw=black] (10,-5) rectangle (11,-6);\node at(10.5,-5.5){\tiny\( \bullet \)};
\filldraw[red,draw=black] (11,-5) rectangle (12,-6);\node at(11.5,-5.5){\tiny\( \bullet \)};
\filldraw[red,draw=black] (12,-5) rectangle (13,-6);\node at(12.5,-5.5){\tiny\( \bullet \)};
\filldraw[red,draw=black] (13,-5) rectangle (14,-6);\node at(13.5,-5.5){\tiny\( \bullet \)};
\filldraw[red,draw=black] (14,-5) rectangle (15,-6);\node at(14.5,-5.5){\tiny\( \bullet \)};
\draw (0,-6) rectangle (1,-7);
\draw (1,-6) rectangle (2,-7);
\draw (2,-6) rectangle (3,-7);
\draw (3,-6) rectangle (4,-7);
\draw (4,-6) rectangle (5,-7);
\draw (5,-6) rectangle (6,-7);
\draw (6,-6) rectangle (7,-7);
\draw (7,-6) rectangle (8,-7);
\draw (8,-6) rectangle (9,-7);
\draw (9,-6) rectangle (10,-7);
\filldraw[blue!60,draw=black] (10,-6) rectangle (11,-7);\node at(10.5,-6.5){\tiny\( \bullet \)};
\filldraw[red,draw=black] (11,-6) rectangle (12,-7);\node at(11.5,-6.5){\tiny\( \bullet \)};
\filldraw[red,draw=black] (12,-6) rectangle (13,-7);\node at(12.5,-6.5){\tiny\( \bullet \)};
\filldraw[red,draw=black] (13,-6) rectangle (14,-7);\node at(13.5,-6.5){\tiny\( \bullet \)};
\filldraw[red,draw=black] (14,-6) rectangle (15,-7);\node at(14.5,-6.5){\tiny\( \bullet \)};
\draw (0,-7) rectangle (1,-8);
\draw (1,-7) rectangle (2,-8);
\draw (2,-7) rectangle (3,-8);
\draw (3,-7) rectangle (4,-8);
\draw (4,-7) rectangle (5,-8);
\draw (5,-7) rectangle (6,-8);
\draw (6,-7) rectangle (7,-8);
\draw (7,-7) rectangle (8,-8);
\draw (8,-7) rectangle (9,-8);
\draw (9,-7) rectangle (10,-8);
\draw (10,-7) rectangle (11,-8);
\draw (11,-7) rectangle (12,-8);
\filldraw[red,draw=black] (12,-7) rectangle (13,-8);\node at(12.5,-7.5){\tiny\( \bullet \)};
\filldraw[red,draw=black] (13,-7) rectangle (14,-8);\node at(13.5,-7.5){\tiny\( \bullet \)};
\filldraw[red,draw=black] (14,-7) rectangle (15,-8);\node at(14.5,-7.5){\tiny\( \bullet \)};
\draw (0,-8) rectangle (1,-9);
\draw (1,-8) rectangle (2,-9);
\draw (2,-8) rectangle (3,-9);
\draw (3,-8) rectangle (4,-9);
\draw (4,-8) rectangle (5,-9);
\draw (5,-8) rectangle (6,-9);
\draw (6,-8) rectangle (7,-9);
\draw (7,-8) rectangle (8,-9);
\draw (8,-8) rectangle (9,-9);
\draw (9,-8) rectangle (10,-9);
\draw (10,-8) rectangle (11,-9);
\draw (11,-8) rectangle (12,-9);
\draw (12,-8) rectangle (13,-9);
\filldraw[red,draw=black] (13,-8) rectangle (14,-9);\node at(13.5,-8.5){\tiny\( \bullet \)};
\filldraw[red,draw=black] (14,-8) rectangle (15,-9);\node at(14.5,-8.5){\tiny\( \bullet \)};
\draw (0,-9) rectangle (1,-10);
\draw (1,-9) rectangle (2,-10);
\draw (2,-9) rectangle (3,-10);
\draw (3,-9) rectangle (4,-10);
\draw (4,-9) rectangle (5,-10);
\draw (5,-9) rectangle (6,-10);
\draw (6,-9) rectangle (7,-10);
\draw (7,-9) rectangle (8,-10);
\draw (8,-9) rectangle (9,-10);
\draw (9,-9) rectangle (10,-10);
\draw (10,-9) rectangle (11,-10);
\draw (11,-9) rectangle (12,-10);
\draw (12,-9) rectangle (13,-10);
\draw (13,-9) rectangle (14,-10);
\filldraw[blue!60,draw=black] (14,-9) rectangle (15,-10);\node at(14.5,-9.5){\tiny\( \bullet \)};
\draw (0,-10) rectangle (1,-11);
\draw (1,-10) rectangle (2,-11);
\draw (2,-10) rectangle (3,-11);
\draw (3,-10) rectangle (4,-11);
\draw (4,-10) rectangle (5,-11);
\draw (5,-10) rectangle (6,-11);
\draw (6,-10) rectangle (7,-11);
\draw (7,-10) rectangle (8,-11);
\draw (8,-10) rectangle (9,-11);
\draw (9,-10) rectangle (10,-11);
\draw (10,-10) rectangle (11,-11);
\draw (11,-10) rectangle (12,-11);
\draw (12,-10) rectangle (13,-11);
\draw (13,-10) rectangle (14,-11);
\filldraw[blue!60,draw=black] (14,-10) rectangle (15,-11);\node at(14.5,-10.5){\tiny\( \bullet \)};
\draw (0,-11) rectangle (1,-12);
\draw (1,-11) rectangle (2,-12);
\draw (2,-11) rectangle (3,-12);
\draw (3,-11) rectangle (4,-12);
\draw (4,-11) rectangle (5,-12);
\draw (5,-11) rectangle (6,-12);
\draw (6,-11) rectangle (7,-12);
\draw (7,-11) rectangle (8,-12);
\draw (8,-11) rectangle (9,-12);
\draw (9,-11) rectangle (10,-12);
\draw (10,-11) rectangle (11,-12);
\draw (11,-11) rectangle (12,-12);
\draw (12,-11) rectangle (13,-12);
\draw (13,-11) rectangle (14,-12);
\draw (14,-11) rectangle (15,-12);
\draw (0,-12) rectangle (1,-13);
\draw (1,-12) rectangle (2,-13);
\draw (2,-12) rectangle (3,-13);
\draw (3,-12) rectangle (4,-13);
\draw (4,-12) rectangle (5,-13);
\draw (5,-12) rectangle (6,-13);
\draw (6,-12) rectangle (7,-13);
\draw (7,-12) rectangle (8,-13);
\draw (8,-12) rectangle (9,-13);
\draw (9,-12) rectangle (10,-13);
\draw (10,-12) rectangle (11,-13);
\draw (11,-12) rectangle (12,-13);
\draw (12,-12) rectangle (13,-13);
\draw (13,-12) rectangle (14,-13);
\draw (14,-12) rectangle (15,-13);
\draw (0,-13) rectangle (1,-14);
\draw (1,-13) rectangle (2,-14);
\draw (2,-13) rectangle (3,-14);
\draw (3,-13) rectangle (4,-14);
\draw (4,-13) rectangle (5,-14);
\draw (5,-13) rectangle (6,-14);
\draw (6,-13) rectangle (7,-14);
\draw (7,-13) rectangle (8,-14);
\draw (8,-13) rectangle (9,-14);
\draw (9,-13) rectangle (10,-14);
\draw (10,-13) rectangle (11,-14);
\draw (11,-13) rectangle (12,-14);
\draw (12,-13) rectangle (13,-14);
\draw (13,-13) rectangle (14,-14);
\draw (14,-13) rectangle (15,-14);
\draw (0,-14) rectangle (1,-15);
\draw (1,-14) rectangle (2,-15);
\draw (2,-14) rectangle (3,-15);
\draw (3,-14) rectangle (4,-15);
\draw (4,-14) rectangle (5,-15);
\draw (5,-14) rectangle (6,-15);
\draw (6,-14) rectangle (7,-15);
\draw (7,-14) rectangle (8,-15);
\draw (8,-14) rectangle (9,-15);
\draw (9,-14) rectangle (10,-15);
\draw (10,-14) rectangle (11,-15);
\draw (11,-14) rectangle (12,-15);
\draw (12,-14) rectangle (13,-15);
\draw (13,-14) rectangle (14,-15);
\draw (14,-14) rectangle (15,-15);
\node at(7.7,-16){\scriptsize\( \tiny{\textbf{Fig.~1.} }\) };
\draw[black,line width=0.5pt] (0,0)--(15,-15);
%
\draw[orange,line width=0.8pt] (3.5,-2)--(3.5,-3.5);
\draw[orange,line width=0.8pt] (3.5,-3.5)--(5,-3.5);
\draw[orange,line width=0.8pt] (5.5,-4)--(5.5,-5.5);
\draw[orange,line width=0.8pt] (5.5,-5.5)--(9,-5.5);
\draw[orange,line width=0.8pt] (9.5,-6)--(9.5,-9.5);
\draw[orange,line width=0.8pt] (9.5,-9.5)--(14,-9.5);
%
\draw[green,line width=0.8pt] (10.5,-7)--(10.5,-10.5);
\draw[green,line width=0.8pt] (10.5,-10.5)--(14,-10.5);
\draw[green,line width=0.8pt] (14.5,-11)--(14.5,-14.5);
\draw[green,line width=0.8pt] (14.5,-14.5)--(15,-14.5);
\draw[black,line width=0.8pt,->] (15,-9.5)--(15.8,-9.5);
\node at (16.1,-9.5){ \tiny{\textbf{$z$}} };
\draw[black,line width=0.8pt,->] (15,-10.5)--(15.8,-10.5);
\node at (16.9,-10.5){ \tiny{\textbf{$z+1$}} };
\draw[black,line width=0.8pt,->] (15,-5.5)--(15.8,-5.5);
\node at (16.3,-5.5){ \tiny{\textbf{$y_1$}} };
\end{tikzpicture}
\hspace{2cm}
\begin{tikzpicture}[scale=.25,line width=0.5pt,baseline=(a.base)]
\draw (0,0) rectangle (1,-1);
\draw (1,0) rectangle (2,-1);
\filldraw[red,draw=black] (2,0) rectangle (3,-1);\node at(2.5,-0.5){\tiny\( \bullet \)};
\filldraw[red,draw=black] (3,0) rectangle (4,-1);\node at(3.5,-0.5){\tiny\( \bullet \)};
\filldraw[red,draw=black] (4,0) rectangle (5,-1);\node at(4.5,-0.5){\tiny\( \bullet \)};
\filldraw[red,draw=black] (5,0) rectangle (6,-1);\node at(5.5,-0.5){\tiny\( \bullet \)};
\filldraw[red,draw=black] (6,0) rectangle (7,-1);\node at(6.5,-0.5){\tiny\( \bullet \)};
\filldraw[red,draw=black] (7,0) rectangle (8,-1);\node at(7.5,-0.5){\tiny\( \bullet \)};
\filldraw[red,draw=black] (8,0) rectangle (9,-1);\node at(8.5,-0.5){\tiny\( \bullet \)};
\filldraw[red,draw=black] (9,0) rectangle (10,-1);\node at(9.5,-0.5){\tiny\( \bullet \)};
\filldraw[red,draw=black] (10,0) rectangle (11,-1);\node at(10.5,-0.5){\tiny\( \bullet \)};
\filldraw[red,draw=black] (11,0) rectangle (12,-1);\node at(11.5,-0.5){\tiny\( \bullet \)};
\filldraw[red,draw=black] (12,0) rectangle (13,-1);\node at(12.5,-0.5){\tiny\( \bullet \)};
\filldraw[red,draw=black] (13,0) rectangle (14,-1);\node at(13.5,-0.5){\tiny\( \bullet \)};
\filldraw[red,draw=black] (14,0) rectangle (15,-1);\node at(14.5,-0.5){\tiny\( \bullet \)};
\draw (0,-1) rectangle (1,-2);
\draw (1,-1) rectangle (2,-2);
\draw (2,-1) rectangle (3,-2);
\filldraw[red,draw=black] (3,-1) rectangle (4,-2);\node at(3.5,-1.5){\tiny\( \bullet \)};
\filldraw[red,draw=black] (4,-1) rectangle (5,-2);\node at(4.5,-1.5){\tiny\( \bullet \)};
\filldraw[red,draw=black] (5,-1) rectangle (6,-2);\node at(5.5,-1.5){\tiny\( \bullet \)};
\filldraw[red,draw=black] (6,-1) rectangle (7,-2);\node at(6.5,-1.5){\tiny\( \bullet \)};
\filldraw[red,draw=black] (7,-1) rectangle (8,-2);\node at(7.5,-1.5){\tiny\( \bullet \)};
\filldraw[red,draw=black] (8,-1) rectangle (9,-2);\node at(8.5,-1.5){\tiny\( \bullet \)};
\filldraw[red,draw=black] (9,-1) rectangle (10,-2);\node at(9.5,-1.5){\tiny\( \bullet \)};
\filldraw[red,draw=black] (10,-1) rectangle (11,-2);\node at(10.5,-1.5){\tiny\( \bullet \)};
\filldraw[red,draw=black] (11,-1) rectangle (12,-2);\node at(11.5,-1.5){\tiny\( \bullet \)};
\filldraw[red,draw=black] (12,-1) rectangle (13,-2);\node at(12.5,-1.5){\tiny\( \bullet \)};
\filldraw[red,draw=black] (13,-1) rectangle (14,-2);\node at(13.5,-1.5){\tiny\( \bullet \)};
\filldraw[red,draw=black] (14,-1) rectangle (15,-2);\node at(14.5,-1.5){\tiny\( \bullet \)};
\draw (0,-2) rectangle (1,-3);
\draw (1,-2) rectangle (2,-3);
\draw (2,-2) rectangle (3,-3);
\draw (3,-2) rectangle (4,-3);
\filldraw[blue!60,draw=black] (4,-2) rectangle (5,-3);\node at(4.5,-2.5){\tiny\( \bullet \)};
\filldraw[blue!60,draw=black] (5,-2) rectangle (6,-3);\node at(5.5,-2.5){\tiny\( \bullet \)};
\filldraw[red,draw=black] (6,-2) rectangle (7,-3);\node at(6.5,-2.5){\tiny\( \bullet \)};
\filldraw[red,draw=black] (7,-2) rectangle (8,-3);\node at(7.5,-2.5){\tiny\( \bullet \)};
\filldraw[red,draw=black] (8,-2) rectangle (9,-3);\node at(8.5,-2.5){\tiny\( \bullet \)};
\filldraw[red,draw=black] (9,-2) rectangle (10,-3);\node at(9.5,-2.5){\tiny\( \bullet \)};
\filldraw[red,draw=black] (10,-2) rectangle (11,-3);\node at(10.5,-2.5){\tiny\( \bullet \)};
\filldraw[red,draw=black] (11,-2) rectangle (12,-3);\node at(11.5,-2.5){\tiny\( \bullet \)};
\filldraw[red,draw=black] (12,-2) rectangle (13,-3);\node at(12.5,-2.5){\tiny\( \bullet \)};
\filldraw[red,draw=black] (13,-2) rectangle (14,-3);\node at(13.5,-2.5){\tiny\( \bullet \)};
\filldraw[red,draw=black] (14,-2) rectangle (15,-3);\node at(14.5,-2.5){\tiny\( \bullet \)};
\draw (0,-3) rectangle (1,-4);
\draw (1,-3) rectangle (2,-4);
\draw (2,-3) rectangle (3,-4);
\draw (3,-3) rectangle (4,-4);
\draw (4,-3) rectangle (5,-4);
\draw (5,-3) rectangle (6,-4);
\filldraw[blue!60,draw=black] (6,-3) rectangle (7,-4);\node at(6.5,-3.5){\tiny\( \bullet \)};
\filldraw[red,draw=black] (7,-3) rectangle (8,-4);\node at(7.5,-3.5){\tiny\( \bullet \)};
\filldraw[red,draw=black] (8,-3) rectangle (9,-4);\node at(8.5,-3.5){\tiny\( \bullet \)};
\filldraw[red,draw=black] (9,-3) rectangle (10,-4);\node at(9.5,-3.5){\tiny\( \bullet \)};
\filldraw[red,draw=black] (10,-3) rectangle (11,-4);\node at(10.5,-3.5){\tiny\( \bullet \)};
\filldraw[red,draw=black] (11,-3) rectangle (12,-4);\node at(11.5,-3.5){\tiny\( \bullet \)};
\filldraw[red,draw=black] (12,-3) rectangle (13,-4);\node at(12.5,-3.5){\tiny\( \bullet \)};
\filldraw[red,draw=black] (13,-3) rectangle (14,-4);\node at(13.5,-3.5){\tiny\( \bullet \)};
\filldraw[red,draw=black] (14,-3) rectangle (15,-4);\node at(14.5,-3.5){\tiny\( \bullet \)};
\draw (0,-4) rectangle (1,-5);
\draw (1,-4) rectangle (2,-5);
\draw (2,-4) rectangle (3,-5);
\draw (3,-4) rectangle (4,-5);
\draw (4,-4) rectangle (5,-5);
\draw (5,-4) rectangle (6,-5);
\draw (6,-4) rectangle (7,-5);
\filldraw[red,draw=black] (7,-4) rectangle (8,-5);\node at(7.5,-4.5){\tiny\( \bullet \)};
\filldraw[red,draw=black] (8,-4) rectangle (9,-5);\node at(8.5,-4.5){\tiny\( \bullet \)};
\filldraw[red,draw=black] (9,-4) rectangle (10,-5);\node at(9.5,-4.5){\tiny\( \bullet \)};
\filldraw[red,draw=black] (10,-4) rectangle (11,-5);\node at(10.5,-4.5){\tiny\( \bullet \)};
\filldraw[red,draw=black] (11,-4) rectangle (12,-5);\node at(11.5,-4.5){\tiny\( \bullet \)};
\filldraw[red,draw=black] (12,-4) rectangle (13,-5);\node at(12.5,-4.5){\tiny\( \bullet \)};
\filldraw[red,draw=black] (13,-4) rectangle (14,-5);\node at(13.5,-4.5){\tiny\( \bullet \)};
\filldraw[red,draw=black] (14,-4) rectangle (15,-5);\node at(14.5,-4.5){\tiny\( \bullet \)};
\draw (0,-5) rectangle (1,-6);
\draw (1,-5) rectangle (2,-6);
\draw (2,-5) rectangle (3,-6);
\draw (3,-5) rectangle (4,-6);
\draw (4,-5) rectangle (5,-6);
\draw (5,-5) rectangle (6,-6);
\draw (6,-5) rectangle (7,-6);
\draw (7,-5) rectangle (8,-6);
\draw (8,-5) rectangle (9,-6);
\filldraw[blue!60,draw=black] (9,-5) rectangle (10,-6);\node at(9.5,-5.5){\tiny\( \bullet \)};
\filldraw[red,draw=black] (10,-5) rectangle (11,-6);\node at(10.5,-5.5){\tiny\( \bullet \)};
\filldraw[red,draw=black] (11,-5) rectangle (12,-6);\node at(11.5,-5.5){\tiny\( \bullet \)};
\filldraw[red,draw=black] (12,-5) rectangle (13,-6);\node at(12.5,-5.5){\tiny\( \bullet \)};
\filldraw[red,draw=black] (13,-5) rectangle (14,-6);\node at(13.5,-5.5){\tiny\( \bullet \)};
\filldraw[red,draw=black] (14,-5) rectangle (15,-6);\node at(14.5,-5.5){\tiny\( \bullet \)};
\draw (0,-6) rectangle (1,-7);
\draw (1,-6) rectangle (2,-7);
\draw (2,-6) rectangle (3,-7);
\draw (3,-6) rectangle (4,-7);
\draw (4,-6) rectangle (5,-7);
\draw (5,-6) rectangle (6,-7);
\draw (6,-6) rectangle (7,-7);
\draw (7,-6) rectangle (8,-7);
\draw (8,-6) rectangle (9,-7);
\draw (9,-6) rectangle (10,-7);
\filldraw[blue!60,draw=black] (10,-6) rectangle (11,-7);\node at(10.5,-6.5){\tiny\( \bullet \)};
\filldraw[red,draw=black] (11,-6) rectangle (12,-7);\node at(11.5,-6.5){\tiny\( \bullet \)};
\filldraw[red,draw=black] (12,-6) rectangle (13,-7);\node at(12.5,-6.5){\tiny\( \bullet \)};
\filldraw[red,draw=black] (13,-6) rectangle (14,-7);\node at(13.5,-6.5){\tiny\( \bullet \)};
\filldraw[red,draw=black] (14,-6) rectangle (15,-7);\node at(14.5,-6.5){\tiny\( \bullet \)};
\draw (0,-7) rectangle (1,-8);
\draw (1,-7) rectangle (2,-8);
\draw (2,-7) rectangle (3,-8);
\draw (3,-7) rectangle (4,-8);
\draw (4,-7) rectangle (5,-8);
\draw (5,-7) rectangle (6,-8);
\draw (6,-7) rectangle (7,-8);
\draw (7,-7) rectangle (8,-8);
\draw (8,-7) rectangle (9,-8);
\draw (9,-7) rectangle (10,-8);
\draw (10,-7) rectangle (11,-8);
\draw (11,-7) rectangle (12,-8);
\filldraw[red,draw=black] (12,-7) rectangle (13,-8);\node at(12.5,-7.5){\tiny\( \bullet \)};
\filldraw[red,draw=black] (13,-7) rectangle (14,-8);\node at(13.5,-7.5){\tiny\( \bullet \)};
\filldraw[red,draw=black] (14,-7) rectangle (15,-8);\node at(14.5,-7.5){\tiny\( \bullet \)};
\draw (0,-8) rectangle (1,-9);
\draw (1,-8) rectangle (2,-9);
\draw (2,-8) rectangle (3,-9);
\draw (3,-8) rectangle (4,-9);
\draw (4,-8) rectangle (5,-9);
\draw (5,-8) rectangle (6,-9);
\draw (6,-8) rectangle (7,-9);
\draw (7,-8) rectangle (8,-9);
\draw (8,-8) rectangle (9,-9);
\draw (9,-8) rectangle (10,-9);
\draw (10,-8) rectangle (11,-9);
\draw (11,-8) rectangle (12,-9);
\draw (12,-8) rectangle (13,-9);
\filldraw[red,draw=black] (13,-8) rectangle (14,-9);\node at(13.5,-8.5){\tiny\( \bullet \)};
\filldraw[red,draw=black] (14,-8) rectangle (15,-9);\node at(14.5,-8.5){\tiny\( \bullet \)};
\draw (0,-9) rectangle (1,-10);
\draw (1,-9) rectangle (2,-10);
\draw (2,-9) rectangle (3,-10);
\draw (3,-9) rectangle (4,-10);
\draw (4,-9) rectangle (5,-10);
\draw (5,-9) rectangle (6,-10);
\draw (6,-9) rectangle (7,-10);
\draw (7,-9) rectangle (8,-10);
\draw (8,-9) rectangle (9,-10);
\draw (9,-9) rectangle (10,-10);
\draw (10,-9) rectangle (11,-10);
\draw (11,-9) rectangle (12,-10);
\draw (12,-9) rectangle (13,-10);
\draw (13,-9) rectangle (14,-10);
\filldraw[blue!60,draw=black] (14,-9) rectangle (15,-10);\node at(14.5,-9.5){\tiny\( \bullet \)};
\draw (0,-10) rectangle (1,-11);
\draw (1,-10) rectangle (2,-11);
\draw (2,-10) rectangle (3,-11);
\draw (3,-10) rectangle (4,-11);
\draw (4,-10) rectangle (5,-11);
\draw (5,-10) rectangle (6,-11);
\draw (6,-10) rectangle (7,-11);
\draw (7,-10) rectangle (8,-11);
\draw (8,-10) rectangle (9,-11);
\draw (9,-10) rectangle (10,-11);
\draw (10,-10) rectangle (11,-11);
\draw (11,-10) rectangle (12,-11);
\draw (12,-10) rectangle (13,-11);
\draw (13,-10) rectangle (14,-11);
\filldraw[blue!60,draw=black] (14,-10) rectangle (15,-11);\node at(14.5,-10.5){\tiny\( \bullet \)};
\draw (0,-11) rectangle (1,-12);
\draw (1,-11) rectangle (2,-12);
\draw (2,-11) rectangle (3,-12);
\draw (3,-11) rectangle (4,-12);
\draw (4,-11) rectangle (5,-12);
\draw (5,-11) rectangle (6,-12);
\draw (6,-11) rectangle (7,-12);
\draw (7,-11) rectangle (8,-12);
\draw (8,-11) rectangle (9,-12);
\draw (9,-11) rectangle (10,-12);
\draw (10,-11) rectangle (11,-12);
\draw (11,-11) rectangle (12,-12);
\draw (12,-11) rectangle (13,-12);
\draw (13,-11) rectangle (14,-12);
\draw (14,-11) rectangle (15,-12);
\draw (0,-12) rectangle (1,-13);
\draw (1,-12) rectangle (2,-13);
\draw (2,-12) rectangle (3,-13);
\draw (3,-12) rectangle (4,-13);
\draw (4,-12) rectangle (5,-13);
\draw (5,-12) rectangle (6,-13);
\draw (6,-12) rectangle (7,-13);
\draw (7,-12) rectangle (8,-13);
\draw (8,-12) rectangle (9,-13);
\draw (9,-12) rectangle (10,-13);
\draw (10,-12) rectangle (11,-13);
\draw (11,-12) rectangle (12,-13);
\draw (12,-12) rectangle (13,-13);
\draw (13,-12) rectangle (14,-13);
\draw (14,-12) rectangle (15,-13);
\draw (0,-13) rectangle (1,-14);
\draw (1,-13) rectangle (2,-14);
\draw (2,-13) rectangle (3,-14);
\draw (3,-13) rectangle (4,-14);
\draw (4,-13) rectangle (5,-14);
\draw (5,-13) rectangle (6,-14);
\draw (6,-13) rectangle (7,-14);
\draw (7,-13) rectangle (8,-14);
\draw (8,-13) rectangle (9,-14);
\draw (9,-13) rectangle (10,-14);
\draw (10,-13) rectangle (11,-14);
\draw (11,-13) rectangle (12,-14);
\draw (12,-13) rectangle (13,-14);
\draw (13,-13) rectangle (14,-14);
\draw (14,-13) rectangle (15,-14);
\draw (0,-14) rectangle (1,-15);
\draw (1,-14) rectangle (2,-15);
\draw (2,-14) rectangle (3,-15);
\draw (3,-14) rectangle (4,-15);
\draw (4,-14) rectangle (5,-15);
\draw (5,-14) rectangle (6,-15);
\draw (6,-14) rectangle (7,-15);
\draw (7,-14) rectangle (8,-15);
\draw (8,-14) rectangle (9,-15);
\draw (9,-14) rectangle (10,-15);
\draw (10,-14) rectangle (11,-15);
\draw (11,-14) rectangle (12,-15);
\draw (12,-14) rectangle (13,-15);
\draw (13,-14) rectangle (14,-15);
\draw (14,-14) rectangle (15,-15);
\node at(7.7,-16){\scriptsize\( \tiny{\textbf{Fig.~2.} }\) };
\draw[black,line width=0.5pt] (0,0)--(15,-15);
%
\draw[orange,line width=0.8pt] (9.5,-6)--(9.5,-9.5);
\draw[orange,line width=0.8pt] (9.5,-9.5)--(14,-9.5);
%
\draw[green,line width=0.8pt] (3.5,-2)--(3.5,-3.5);
\draw[green,line width=0.8pt] (3.5,-3.5)--(6,-3.5);
\draw[green,line width=0.8pt] (6.5,-4)--(6.5,-6.5);
\draw[green,line width=0.8pt] (6.5,-6.5)--(10,-6.5);
\draw[green,line width=0.8pt] (10.5,-7)--(10.5,-10.5);
\draw[green,line width=0.8pt] (10.5,-10.5)--(14,-10.5);
\draw[green,line width=0.8pt] (14.5,-11)--(14.5,-14.5);
\draw[green,line width=0.8pt] (14.5,-14.5)--(15,-14.5);
\draw[black,line width=0.8pt,->] (15,-9.5)--(15.8,-9.5);
\node at (16.1,-9.5){ \tiny{\textbf{$z$}} };
\draw[black,line width=0.8pt,->] (15,-10.5)--(15.8,-10.5);
\node at (16.9,-10.5){ \tiny{\textbf{$z+1$}} };
\draw[black,line width=0.8pt,->] (15,-5.5)--(15.8,-5.5);
\node at (16.3,-5.5){ \tiny{\textbf{$y_2$}} };
\end{tikzpicture}
\end{equation*}
Here in Fig.~1, the orange (resp. green) line represents the bounce path $(2,4,6,z)$ (resp. $(7,z+1, 15)$) containing $z$ (resp. $z+1$) and
\begin{align*}
{\rm top}_{\dkmu}(z+1) =&\ 7 > 2 = {\rm top}_{\dkmu}(z).
\end{align*}
In Fig.~2, the orange (resp. green) line represents the bounce path $(6,z)$ (resp. $(2,4,7,z+1, 15)$) containing $z$ (resp. $z+1$) and
\begin{align*}
\textup{top}_{\Delta^{k}(\mu)}(z)=&\ 6>2 = \textup{top}_{\Delta^{k}(\mu)}(z+1).
\end{align*}
The blue part makes $\Delta^{k}(\mu)$ satisfing (I)  and (II), respectively.
\end{enumerate}
\mlabel{re:factroot}
\end{remark}

To study the lowering operator acting on closed $k$-Schur Katalan functions, we generalize the concept of closed $k$-Schur Katalan function as follows.

\begin{defn}
For $\mu\in\tilde{\text{P}}_{\ell}^{k}$, the {\bf generalized closed $k$-Schur Katalan function} is
\begin{equation*}
\fg{\mu}{k}:=K(\Delta^{k}(\mu);\Delta^{k}(\mu);\mu).
\end{equation*}
\mlabel{def:geneCK}
\end{defn}

The following result will be needed later.

\begin{lemma}
Let $\lambda\in\pkl$, $z\in[\ell]$ and $\bott:=\bott_\lambda$ be the bottom of $\dkl$. Denote $\mu:=\mu_{\lambda,z}:= \lambda-\epsilon_z\in\tpkl$.
\begin{enumerate}
\item
If $z\in[\bott+1,\ell]$, then
$
L_{z}\fg{\lambda}{k} = \fg{\mu}{k}.
$
\mlabel{it:LamToMu1}

\item
If $z\in[\bott]$, then
$
L_{z}\fg{\lambda}{k}=\fg{\mu}{k}-
L_{\down{\dkl}{z}}\fg{\mu}{k}+
L_{\down{\dkl}{z}}\fg{\lambda}{k}.
$
\mlabel{it:LamToMu2}
\end{enumerate}
\mlabel{lem:LamToMu}
\end{lemma}

\begin{proof}
(\mref{it:LamToMu1}) Since $z\in[\bott+1,\ell]$, there is no root in the row $z$ in $\Delta^{k}(\lambda)$.
It follows from~(\mref{eq:dkl}) that
$k-\lambda_{z}+z \geq \ell.$ Further,
$$k-\mu_{z}+z = k-(\lambda_{z}-1)+z > k-\lambda_{z}+z \geq \ell,$$
which means that the row $z$ has no root in the root ideal $\Delta^{k}(\mu)$.
Hence $\Delta^{k}(\lambda) = \Delta^{k}(\mu)$ and so
\begin{align*}
L_z \fg{\lambda}{k}
=&\ K(\dkl;\dkl;\mu)\hspace{1cm}(\text{by~(\ref{eq:LamToMu-a}) and $\mu=\lambda-\epsilon_z$})\\
=&\ K(\dkmu;\dkmu;\mu)\\
=&\ \fg{\mu}{k} \hspace{1cm} (\text{by Definition~\ref{def:geneCK}}).
\end{align*}

(\mref{it:LamToMu2})
If $z\in[\bott]$, then $\beta:=(z,\text{down}_{\Delta^{k}(\lambda)}(z))$ is a removable root in $\Delta^{k}(\lambda)$ by Remark~\mref{re:factroot}~(c), and so $\dkl\setminus\beta = \dkmu$. In terms of~(\mref{eq:Lroot}),
\begin{align*}
L(\dkmu) = L(\dkl\setminus\beta)
= L(\dkl)\setminus L(\beta)=L(\dkl)\setminus \text{down}_{\Delta^{k}(\lambda)}(z).
\end{align*}
Hence
\begin{align*}
L_z\fg{\lambda}{k} =&\ K(\dkl;\dkl;\mu)\hspace{1cm}(\text{by~(\ref{eq:LamToMu-a}) and $\mu=\lambda-\epsilon_z$}) \\
=&\ K(\dkl\setminus\beta;\dkl;\mu)+ K(\dkl;\dkl;\mu+\varepsilon_{\beta})\hspace{1cm} (\text{by Lemma~\ref{lem:relk}~\ref{it:relk1}}) \\
=&\ K(\dkl\setminus\beta;\dkl;\mu)+ K\big(\dkl;\dkl;\lambda-\epsilon_z+\epsilon_z-\epsilon_{\down{\dkl}{z}}\big)\\
=&\ K(\dkl\setminus\beta;\dkl;\mu) + L_{\down{\dkl}{z}}K(\dkl;\dkl;\lambda)\\
&\ \hspace{5cm} (\text{by Remark~\ref{re:BMS}~(a) for the second summand})\\
=&\ K(\dkl\setminus\beta;\dkl;\mu) + L_{\down{\dkl}{z}}\fg{\lambda}{k}\\
&\ \hspace{4cm} (\text{by Definition~\ref{defn:cKataFunc} for the second summand})\\
=&\ K(\dkl\setminus\beta;L(\dkl)\setminus\down{\dkl}{z};\mu)\\
&\ -
K(\dkl\setminus\beta;L(\dkl)\setminus\down{\dkl}{z};\mu-\epsilon_{\down{\dkl}{z}}) + L_{\down{\dkl}{z}}\fg{\lambda}{k}\\
&\ \hspace{4cm} (\text{by Lemma~\ref{lem:relk}~\ref{it:relk3} for the first summand})\\
=&\ K(\dkmu;\dkmu;\mu)-K(\dkmu;\dkmu;\mu-\epsilon_{\down{\dkl}{z}})+ L_{\down{\dkl}{z}}\fg{\lambda}{k}\\
&\ \hspace{1cm} (\text{by $\dkl\setminus\beta = \dkmu$ and $L(\dkl)\setminus\down{\dkl}{z} = L(\dkmu)$})\\
=&\ K(\dkmu;\dkmu;\mu)-L_{\down{\dkl}{z}}K(\dkmu;\dkmu;\mu)+ L_{\down{\dkl}{z}}\fg{\lambda}{k}\\
&\ \hspace{5cm} (\text{by Remark~\ref{re:BMS}~(a) for the second summand})\\
=&\ \fg{\mu}{k}-
L_{\down{\dkl}{z}}\fg{\mu}{k}+
L_{\down{\dkl}{z}}\fg{\lambda}{k} \hspace{1cm} (\text{by Definition~\ref{def:geneCK}}).
\end{align*}
This completes the proof.
\end{proof}

In Lemma~\mref{lem:LamToMu},
if $z = \ell$ or $\lambda_z>\lambda_{z+1}$ with $z\in[\ell-1]$, then $\mu\in{\rm P}^k_\ell$.
Now consider the case of $\lambda_z = \lambda_{z+1}$. Then $\mu\in\tpkl$ and $\mu\notin\pkl$ by
$$\mu_z = \lambda_z -1 = \lambda_{z+1} -1 =\mu_{z+1}-1.$$
The following result is a straightening of $L_z\fg{\lambda}{k}$ in Lemma~\mref{lem:LamToMu}, in the sense that
the subscripts $\mu-\epsilon_{z+1}$ and $\mu+\epsilon_{\textup{up}_{\Delta^{k}(\mu)}
(y+1)}-\epsilon_{z+1}$ are more like partitions, compared to the original $\mu\in \tpkl$.

\begin{lemma}
Let $\mu\in\tpkl$ and $z\in[\ell-1]$ such that $\mu_{z}+1=\mu_{z+1}$ and $\mu_{x}\geq\mu_{x+1}$ for all $x\in[\ell-1]\setminus z$.
\begin{enumerate}
\item If ${\rm top}_{\dkmu}(z+1)> {\rm top}_{\dkmu}(z)$,
then $\fg{\mu}{k}=\fg{\mu-\epsilon_{z+1}}{k}$.
\mlabel{it:strai1}

\item
If $y:=\textup{top}_{\Delta^{k}(\mu)}(z)>
\textup{top}_{\Delta^{k}(\mu)}(z+1)$,
then
$\fg{\mu}{k} =\fg{\mu+\epsilon_{\textup{up}_{\Delta^{k}(\mu)}
(y+1)}-\epsilon_{z+1}}{k}.
$ \mlabel{it:strai2}
\end{enumerate}
\mlabel{lem:strai}
\end{lemma}

Notice that ${\rm top}_{\dkmu}(z+1)\neq {\rm top}_{\dkmu}(z)$ by~(\mref{eq:noteq}).

\begin{proof}[proof of Lemma~\mref{lem:strai}]
By Remark~\mref{re:factroot}~(b), the hypothesis $\mu_{z}+1=\mu_{z+1}$ implies that there is a wall in rows $z,z+1$ in $\dkmu$. Then
\begin{equation}
z\in[\bott-1]\,\text{ or }\,z\in[\bott+1,\ell-1],
\mlabel{eq:strai-tcase}
\end{equation}
where $\bott:=\bott_\mu$ is the bottom of the root ideal $\dkmu$.

(\mref{it:strai1})
Suppose $y:={\rm top}_{\dkmu}(z+1)> {\rm top}_{\dkmu}(z)$. Then the root ideal $\Delta^{k}(\mu)$ has a ceiling in columns $y,y+1$, and $m_{\dkmu}(y) = m_{\dkmu}(y+1)$. According to~(\mref{eq:strai-tcase}), we have the following two cases to consider.

{\bf Case~1.} $z\in[\bott+1,\ell-1]$. Then $\De{k}{\mu} = \De{k}{\mu-\epsilon_{z+1}}$ and so
\begin{align*}
\fg{\mu}{k}
=&\ K(\Delta^{k}(\mu);\Delta^{k}(\mu);\mu) \hspace{1cm} (\text{by Definition~\ref{def:geneCK}})\\
=&\ K(\Delta^{k}(\mu);\Delta^{k}(\mu);\mu-\epsilon_{z+1})\hspace{1cm} (\text{by Lemma~\ref{lem:mirr1}})\\
=&\ K(\Delta^{k}(\mu-\epsilon_{z+1});\Delta^{k}(\mu-\epsilon_{z+1});
\mu-\epsilon_{z+1})\\
=&\ \fg{\mu-\epsilon_{z+1}}{k}\hspace{1cm} (\text{by Definition~\ref{def:geneCK}}).
\end{align*}

{\bf Case~2.} $z\in[\bott-1]$. Then $z+1\in[\bott]$ and further by Remark~\mref{re:factroot}~(c),
\begin{equation}
\beta:=(z+1,\down{\dkmu}{z+1})
\mlabel{eq:beta}
\end{equation}
is a removable root in the root ideal $\dkmu$. Using Remark~\mref{re:factroot}~(a),
$$\dkmu\setminus\beta = \De{k}{\mu-\epsilon_{z+1}}, \quad L(\dkmu)\setminus \down{\dkmu}{z+1} = L(\De{k}{\mu-\epsilon_{z+1}}).$$
Hence
\begin{align*}
\fg{\mu}{k} =&\ K(\dkmu;\dkmu;\mu)\hspace{1cm} (\text{by Definition~\ref{def:geneCK}}) \\
=&\ K(\dkmu;\dkmu;\mu-\epsilon_{z+1})\hspace{1cm} (\text{by Lemma~\ref{lem:mirr1}})\\
=&\ K(\dkmu\setminus\beta;\dkmu;\mu-\epsilon_{z+1})
+K(\dkmu;\dkmu;\mu-\epsilon_{z+1}+\varepsilon_{\beta})\hspace{1cm} (\text{by Lemma~\ref{lem:relk}~\ref{it:relk1}})\\
=&\ K(\dkmu\setminus\beta;\dkmu;\mu-\epsilon_{z+1})
+K(\dkmu;\dkmu;\mu-\epsilon_{\down{\dkmu}{z+1}})\hspace{1cm} (\text{by~(\ref{eq:beta})})\\
=&\ K(\dkmu\setminus\beta;\dkmu;\mu-\epsilon_{z+1})
+L_{\down{\dkmu}{z+1}}K(\dkmu;\dkmu;\mu)\\
& \hspace{6cm} (\text{by Remark~\ref{re:BMS}~(a) for the second summand})\\
=&\ K(\dkmu\setminus\beta;\dkmu;\mu-\epsilon_{z+1})
+L_{\down{\dkmu}{z+1}}\fg{\mu}{k}\\
& \hspace{6cm} (\text{by Definition~\ref{def:geneCK} for the second summand})\\
=&\ K(\dkmu\setminus\beta;L(\dkmu)\setminus\down{\dkmu}{z+1};\mu-\epsilon_{z+1})\\
&\ -K(\dkmu\setminus\beta;L(\dkmu)\setminus\down{\dkmu}{z+1};\mu-\epsilon_{z+1}
-\epsilon_{\down{\dkmu}{z+1}})\\
&\ +L_{\down{\dkmu}{z+1}}\fg{\mu}{k} \hspace{3cm} (\text{by Lemma~\ref{lem:relk}~\ref{it:relk3} for the first summand})\\
=&\ K(\De{k}{\mu-\epsilon_{z+1}};\De{k}{\mu-\epsilon_{z+1}};\mu-\epsilon_{z+1})\\
&\ -K(\De{k}{\mu-\epsilon_{z+1}};\De{k}{\mu-\epsilon_{z+1}};\mu-\epsilon_{z+1}
-\epsilon_{\down{\dkmu}{z+1}})+L_{\down{\dkmu}{z+1}}\fg{\mu}{k}\\
=&\ K(\De{k}{\mu-\epsilon_{z+1}};\De{k}{\mu-\epsilon_{z+1}};\mu-\epsilon_{z+1})\\
&\ -L_{\down{\dkmu}{z+1}}K(\De{k}{\mu-\epsilon_{z+1}};
\De{k}{\mu-\epsilon_{z+1}};\mu-\epsilon_{z+1})+L_{\down{\dkmu}{z+1}}\fg{\mu}{k}\\
&\ \hspace{5cm} (\text{the second summand uses Remark~\ref{re:BMS}~(a)})\\
=&\ \fg{\mu-\epsilon_{z+1}}{k}-L_{\down{\dkmu}{z+1}}\fg{\mu-\epsilon_{z+1}}{k}
+L_{\down{\dkmu}{z+1}}\fg{\mu}{k}\hspace{1cm} (\text{by Definition~\ref{def:geneCK}}),
\end{align*}
which implies that
\begin{align*}
\fg{\mu}{k}-\fg{\mu-\epsilon_{z+1}}{k} = &\ L_{\down{\dkmu}{z+1}}(\fg{\mu}{k}-\fg{\mu-\epsilon_{z+1}}{k})\\
=&\ L^2_{\down{\dkmu}{z+1}}(\fg{\mu}{k}-\fg{\mu-\epsilon_{z+1}}{k})\\
=&\ \cdots = L^m_{\down{\dkmu}{z+1}}(\fg{\mu}{k}-\fg{\mu-\epsilon_{z+1}}{k})\\
=&\ 0\hspace{1cm} (\text{by Proposition~\ref{prop:multin}})
\end{align*}
for some $$m>\mu_{\down{\dkmu}{z+1}}+\ell-\down{\dkmu}{z+1}.$$

(\mref{it:strai2}) Suppose $y:=\textup{top}_{\Delta^{k}(\mu)}(z)>
\textup{top}_{\Delta^{k}(\mu)}(z+1)$.
Then $$\alpha :=(\text{up}_{\Delta^{k}(\mu)}(y+1),y)$$ is an addable root of $\dkmu$, the root $({\rm up}_{\dkmu}(y+1),y+1)$ is removable in $\dkmu$, and $$m_{\dkmu}(x)+1 = m_{\dkmu}(x+1), \quad \forall x\in{\rm path}_{\dkmu}(y,z).$$
Again, we divide our proof into the following two cases by~(\mref{eq:strai-tcase}).

{\bf Case~1.} $z\in[\bott+1,\ell-1]$. Then
\begin{equation}
\begin{aligned}
\Delta^{k}(\mu)\cup\alpha =&\ \Delta^{k}(\mu+\epsilon_{\text{up}_{\Delta^{k}(\mu)}(y+1)}) = \Delta^{k}(\mu+\epsilon_{\text{up}_{\Delta^{k}(\mu)}(y+1)}-\epsilon_{z+1}),\\
L(\Delta^{k}(\mu))\sqcup y =&\ L\big(\Delta^{k}(\mu)\cup\alpha\big) = L\Big(\Delta^{k}(\mu+\epsilon_{\text{up}_{\Delta^{k}(\mu)}(y+1)}-
\epsilon_{z+1})\Big).
\end{aligned}
\mlabel{eq:dmualpha}
\end{equation}
Hence
\allowdisplaybreaks{
\begin{align*}
\fg{\mu}{k}
=&\ K(\Delta^{k}(\mu);\Delta^{k}(\mu);\mu) \hspace{1cm} (\text{by Definition~\ref{def:geneCK}})\\
=&\ K\Big(\Delta^{k}(\mu)\cup\alpha; L(\Delta^{k}(\mu))\sqcup y;\mu+
\epsilon_{\text{up}_{\Delta^{k}(\mu)}(y+1)}-\epsilon_{z+1}\Big)\hspace{1cm} (\text{by Lemma~\ref{lem:mirr2}})\\
=\ & K\bigg(\Delta^{k}(\mu+
\epsilon_{\text{up}_{\Delta^{k}(\mu)}(y+1)}-\epsilon_{z+1}); L\Big( \Delta^{k}(\mu+
\epsilon_{\text{up}_{\Delta^{k}(\mu)}(y+1)}-\epsilon_{z+1}) \Big);\mu+
\epsilon_{\text{up}_{\Delta^{k}(\mu)}(y+1)}-\epsilon_{z+1}\bigg)  \hspace{0.5cm} (\text{by~(\ref{eq:dmualpha})})\\
=&\ \fg{\mu+\epsilon_{\text{up}_{\Delta^{k}(\mu)}(y+1)}-\epsilon_{z+1}}{k}\hspace{1cm} (\text{by Definition~\ref{def:geneCK}}).
\end{align*}
}

{\bf Case~2.} $z\in[\bott-1]$. Then $$\beta:=(z+1,\down{\dkmu}{z+1})$$ is a removable root in $\dkmu$. Further,
\begin{align}
\dkmu\cup\alpha =&\  \De{k}{\mu+\epsilon_{{\rm up}_{\dkmu}(y+1)}},\mlabel{eq:strai-2-2-1}\\
\dkmu\cup\alpha\setminus\beta =&\  \De{k}{\mu+\epsilon_{{\rm up}_{\dkmu}(y+1)}-\epsilon_{z+1}},\label{eq:strai-2-2-2}\\
L(\dkmu)\sqcup y =&\  L\big(\dkmu\cup\alpha\big) = L\big(\De{k}{\mu+\epsilon_{{\rm up}_{\dkmu}(y+1)}}\big),\label{eq:strai-2-2-3}\\
L(\dkmu)\sqcup y\setminus \down{\dkmu}{z+1} =&\  L\big( \dkmu\cup\alpha\setminus\beta \big) = L\big(\De{k}{\mu+\epsilon_{{\rm up}_{\dkmu}(y+1)}-\epsilon_{z+1}}\big).\label{eq:strai-2-2-4}
\end{align}
Hence
\begin{align}
\fg{\mu}{k} =&\ K(\Delta^{k}(\mu);\Delta^{k}(\mu);\mu) \hspace{1cm} (\text{by Definition~\ref{def:geneCK}})\notag\\
=&\ K\Big(\Delta^{k}(\mu)\cup\alpha; L(\Delta^{k}(\mu))\sqcup y;\mu+
\epsilon_{\text{up}_{\Delta^{k}(\mu)}(y+1)}-\epsilon_{z+1}\Big)\hspace{1cm} (\text{by Lemma~\ref{lem:mirr2}})\notag\\
=&\ K\Big(\Delta^{k}(\mu)\cup\alpha\setminus\beta; L(\Delta^{k}(\mu))\sqcup y;\mu+
\epsilon_{\text{up}_{\Delta^{k}(\mu)}(y+1)}-\epsilon_{z+1}\Big)\notag\\
&\ + K\Big(\Delta^{k}(\mu)\cup\alpha; L(\Delta^{k}(\mu))\sqcup y;\mu+
\epsilon_{\text{up}_{\Delta^{k}(\mu)}(y+1)}-\epsilon_{z+1}
+\varepsilon_{\beta}\Big) \hspace{1cm} (\text{by Lemma~\ref{lem:relk}~\ref{it:relk1}})\notag\\
=&\ K\Big(\Delta^{k}(\mu)\cup\alpha\setminus\beta; L(\Delta^{k}(\mu))\sqcup y;\mu+
\epsilon_{\text{up}_{\Delta^{k}(\mu)}(y+1)}-\epsilon_{z+1}\Big)\notag\\
&\ + K\Big(\De{k}{\mu+\epsilon_{{\rm up}_{\dkmu}(y+1)}}; L\big(\De{k}{\mu+\epsilon_{{\rm up}_{\dkmu}(y+1)}}\big);\mu+
\epsilon_{\text{up}_{\Delta^{k}(\mu)}(y+1)}-\epsilon_{\down{\dkmu}{z+1}}\Big)
\notag\\
& \hspace{7cm} (\text{by ~(\ref{eq:strai-2-2-1}) and~(\ref{eq:strai-2-2-3}) for the second summand})\notag\\
=&\ K\Big(\Delta^{k}(\mu)\cup\alpha\setminus\beta; L(\Delta^{k}(\mu))\sqcup y;\mu+
\epsilon_{\text{up}_{\Delta^{k}(\mu)}(y+1)}-\epsilon_{z+1}\Big) + L_{\down{\dkmu}{z+1}}\fg{\mu+
\epsilon_{\text{up}_{\Delta^{k}(\mu)}(y+1)}}{k}\notag\\
& \hspace{4.5cm} (\text{by Remark~\ref{re:BMS}~(a) and Definition~\ref{def:geneCK} for the second summand})\notag\\
=&\ K\Big(\Delta^{k}(\mu)\cup\alpha\setminus\beta; L(\Delta^{k}(\mu))\sqcup y\setminus \down{\dkmu}{z+1};\mu+
\epsilon_{\text{up}_{\Delta^{k}(\mu)}(y+1)}-\epsilon_{z+1}\Big)\notag\\
& - K\Big(\Delta^{k}(\mu)\cup\alpha\setminus\beta; L(\Delta^{k}(\mu))\sqcup y\setminus \down{\dkmu}{z+1};\mu+
\epsilon_{\text{up}_{\Delta^{k}(\mu)}(y+1)}-\epsilon_{z+1}
-\epsilon_{\down{\dkmu}{z+1}}\Big)\notag\\
& + L_{\down{\dkmu}{z+1}}\fg{\mu+
\epsilon_{\text{up}_{\Delta^{k}(\mu)}(y+1)}}{k} \hspace{3.5cm} (\text{by Lemma~\ref{lem:relk}~\ref{it:relk3} for the first summand})\notag\\
=&\ K\Big(\De{k}{\mu+\epsilon_{{\rm up}_{\dkmu}(y+1)}-\epsilon_{z+1}}; \De{k}{\mu+\epsilon_{{\rm up}_{\dkmu}(y+1)}-\epsilon_{z+1}};\mu+
\epsilon_{\text{up}_{\Delta^{k}(\mu)}(y+1)}-\epsilon_{z+1}\Big)\notag\\
& - K\Big(\De{k}{\mu+\epsilon_{{\rm up}_{\dkmu}(y+1)}-\epsilon_{z+1}}; \De{k}{\mu+\epsilon_{{\rm up}_{\dkmu}(y+1)}-\epsilon_{z+1}};\mu+
\epsilon_{\text{up}_{\Delta^{k}(\mu)}(y+1)}-\epsilon_{z+1}
-\epsilon_{\down{\dkmu}{z+1}}\Big)\notag\\
& + L_{\down{\dkmu}{z+1}}\fg{\mu+
\epsilon_{\text{up}_{\Delta^{k}(\mu)}(y+1)}}{k} \hspace{7.5cm} (\text{by~(\ref{eq:strai-2-2-2}) and~(\ref{eq:strai-2-2-4})})\notag\\
=&\ K\Big(\De{k}{\mu+\epsilon_{{\rm up}_{\dkmu}(y+1)}-\epsilon_{z+1}}; \De{k}{\mu+\epsilon_{{\rm up}_{\dkmu}(y+1)}-\epsilon_{z+1}};\mu+
\epsilon_{\text{up}_{\Delta^{k}(\mu)}(y+1)}-\epsilon_{z+1}\Big)\notag\\
& - L_{\down{\dkmu}{z+1}}K\Big(\De{k}{\mu+\epsilon_{{\rm up}_{\dkmu}(y+1)}-\epsilon_{z+1}}; \De{k}{\mu+\epsilon_{{\rm up}_{\dkmu}(y+1)}-\epsilon_{z+1}};\mu+
\epsilon_{\text{up}_{\Delta^{k}(\mu)}(y+1)}-\epsilon_{z+1}\Big)\notag\\
& + L_{\down{\dkmu}{z+1}}\fg{\mu+
\epsilon_{\text{up}_{\Delta^{k}(\mu)}(y+1)}}{k} \hspace{3cm} (\text{by Remark~\ref{re:BMS}~(a) for the second summand})\notag\\
=&\ \fg{\mu+\epsilon_{{\rm up}_{\dkmu}(y+1)}-\epsilon_{z+1}}{k} - L_{\down{\dkmu}{z+1}} \fg{\mu+\epsilon_{{\rm up}_{\dkmu}(y+1)}-\epsilon_{z+1}}{k}+ L_{\down{\dkmu}{z+1}}\fg{\mu+
\epsilon_{\text{up}_{\Delta^{k}(\mu)}(y+1)}}{k}  \hspace{0.3cm} (\text{by Definition~\ref{def:geneCK}}).\mlabel{eq:strai-2-3}
\end{align}
Consider the third summand in~(\mref{eq:strai-2-3}). Denote $\nu:=\mu+
\epsilon_{\text{up}_{\Delta^{k}(\mu)}(y+1)}$. Then
$$\nu_{z}+1= \mu_z+1 = \mu_{z+1} =\nu_{z+1}.$$
It follows from
$y:=\textup{top}_{\Delta^{k}(\mu)}(z)>
\textup{top}_{\Delta^{k}(\mu)}(z+1)$
that
$${\rm top}_{\De{k}{\nu}}(z+1)-1 > {\rm top}_{\De{k}{\nu}}(z), \quad \nu_{x}\geq\nu_{x+1},\quad\forall x\in[\ell-1]\setminus z.$$
Hence
\begin{align*}
\fg{\mu+
\epsilon_{\text{up}_{\Delta^{k}(\mu)}(y+1)}}{k} =&\ \fg{\nu}{k}\hspace{1cm}(\text{by $\nu:=\mu+
\epsilon_{\text{up}_{\Delta^{k}(\mu)}(y+1)}$})\\
=&\ \fg{\nu-\epsilon_{z+1}}{k}\hspace{1cm}(\text{by Lemma~\ref{lem:strai}~\ref{it:strai1}})\\
=&\ \fg{\mu+\epsilon_{{\rm up}_{\dkmu}(y+1)}-\epsilon_{z+1}}{k}.
\end{align*}
Plugging the above equation into the third summand of~(\mref{eq:strai-2-3}) yields killing the last two summands and obtaining
the required result.
\end{proof}

We arrive at our main result in this section.

\begin{theorem}
Let $\lambda\in\pkl$ and $z\in[\ell]$.
Then there exists $\mu\in\pkl$ such that
\begin{equation}
L_{z}\fg{\lambda}{k} =
\left\{
\begin{array}{ll}
 \fg{\mu}{k}, & \text{if $z\in[\bott+1,\ell]$},\\
 \fg{\mu}{k}-
L_{\down{\dkl}{z}}\fg{\mu}{k}+
L_{\down{\dkl}{z}}\fg{\lambda}{k}, & \text{if $z\in[\bott]$}.
\end{array}
\right.
\mlabel{eq:finalform-want}
\end{equation}
Moreover, if $z = \ell$ or $\lambda_z>\lambda_{z+1}$ with $z\in[\ell-1]$, then we can take $\mu := \lambda-\epsilon_z$.
\mlabel{thm:finalform}
\end{theorem}

\begin{proof}
Denote $h:=h_{\lambda,z}\in[0,\ell-z]$ such that
\begin{equation}
\lambda_z = \lambda_{z+1} = \cdots = \lambda_{z+h}\,\text{ and }\ \lambda_{z+h}>\lambda_{z+h+1}\,\text{ if }\, z+h+1\in[\ell].
\mlabel{eq:defnh}
\end{equation}
We prove the first part of the theorem by claiming: for $d\in[0,h]$, there exists $\mu\in\tpkl$ such that
\begin{enumerate}[label=(\roman*)]
\item  $|\mu|<|\lambda|$;
\item  $\mu_{z+d}=\lambda_{z+d}-1$;
\item  $\mu_{x}\geq\mu_{x+1}$ for all $x\in[\ell-1]\setminus (z+d)$;
\item  $\mu_{x} = \lambda_{x}$ for all $x\in[z+d+1,\ell]$ and (\mref{eq:finalform-want}) is valid.
\end{enumerate}

We employ induction on $d\in[0,h]$ to verify the claim. The initial step of $d = 0$ is exactly
Lemma~\mref{lem:LamToMu} by $\mu := \lambda-\epsilon_z$ satisfying
\begin{enumerate}[label=(\roman*)]
\item  $|\mu|<|\lambda|$;
\item  $\mu_{z}=\lambda_{z}-1$;
\item  $\mu_{x}\geq\mu_{x+1}$ for all $x\in[\ell-1]\setminus z$;
\item  $\mu_{x} = \lambda_{x}$ for all $x\in[z+1,\ell]$.
\end{enumerate}

For the inductive step of $d\in[1,h]$, suppose that there exists $\mu'\in\tpkl$ satisfying the claim, that is,
\begin{enumerate}[label=(\roman*)]
\item $|\mu'|<|\lambda|$;
\item $\mu'_{z+d-1}=\lambda_{z+d-1}-1$;
\item $\mu'_{x}\geq\mu'_{x+1}$ for all $x\in[\ell-1]\setminus (z+d-1)$;
\item $\mu'_{x} = \lambda_{x}$ for all $x\in[z+d,\ell]$ and
\begin{equation}
L_{z}\fg{\lambda}{k} =
\left\{
\begin{array}{ll}
 \fg{\mu'}{k}, &\text{if $z\in[\bott+1,\ell]$},\\
 \fg{\mu'}{k}-
L_{\down{\dkl}{z}}\fg{\mu'}{k}+
L_{\down{\dkl}{z}}\fg{\lambda}{k}, &\text{if $z\in[\bott]$}.
\end{array}
\right.
\mlabel{eq:finalform-ind}
\end{equation}
\end{enumerate}
Since
$$\mu'_{z+d-1} + 1 = \lambda_{z+d-1} = \lambda_{z+d}= \mu'_{z+d},\quad \mu'_{x}\geq\mu'_{x+1}, \quad \forall x\in[\ell-1]\setminus (z+d-1),$$
it follows from Remark~\mref{re:factroot}~(d) that we have the following two cases to consider.

{\bf Case~1.} ${\rm top}_{\De{k}{\mu'}}(z+d)> {\rm top}_{\De{k}{\mu'}}(z+d-1)$. By Lemma~\ref{lem:strai}~\ref{it:strai1},
\begin{equation}
\fg{\mu'}{k}
= \fg{\mu'-\epsilon_{z+d}}{k}=:\fg{\mu}{k},
\mlabel{eq:finalform-1}
\end{equation}
where $\mu:=\mu'-\epsilon_{z+d}\in\tpkl$. Moreover,
\begin{enumerate}[label=(\roman*)]
\item $|\mu|<|\mu'|<|\lambda|$;

\item $\mu_{z+d} = \mu'_{z+d}-1 = \lambda_{z+d}-1$;

\item $\mu_{x}\geq\mu_{x+1}$ for all $x\in[\ell-1]\setminus \{z+d-1,z+d\}$, and
\begin{align*}
\mu_{z+d-1} =\mu'_{z+d-1} =&\ \lambda_{z+d-1}-1 =\lambda_{z+d}-1 = \mu'_{z+d}-1= \mu_{z+d},\,\text{ and so }\\
\mu_{x}\geq&\ \mu_{x+1},\quad \forall x\in[\ell-1]\setminus (z+d);
\end{align*}

\item $\mu_{x} = \lambda_{x}$ for all $x\in[z+d+1,\ell]$.
\end{enumerate}
Substituting~(\mref{eq:finalform-1}) into~(\mref{eq:finalform-ind}) gives the required result.

{\bf Case~2.} $y:={\rm top}_{\De{k}{\mu'}}(z+d-1)>{\rm top}_{\De{k}{\mu'}}(z+d)$. Then
\begin{equation}
\mu'_{{\rm up}_{\De{k}{\mu'}}(y+1)-1}>\mu'_{{\rm up}_{\De{k}{\mu'}}(y+1)}.
\mlabel{eq:finalform-2(1)}
\end{equation}
By Lemma~\ref{lem:strai}~\ref{it:strai2},
\begin{equation}
\fg{\mu'}{k} =\fg{\mu'+\epsilon_{\textup{up}_{\Delta^{k}(\mu')}
(y+1)}-\epsilon_{z+d}}{k}=:\fg{\mu}{k},
\mlabel{eq:finalform-2}
\end{equation}
where $\mu:=\mu'+\epsilon_{\textup{up}_{\Delta^{k}(\mu')}
(y+1)}-\epsilon_{z+d}\in\tpkl$. Further,
\begin{enumerate}[label=(\roman*)]
\item $|\mu| = |\mu'|<|\lambda|$;

\item $\mu_{z+d} = \mu'_{z+d}-1 = \lambda_{z+d}-1$;

\item $\mu_x\geq\mu_{x+1}$ for all $x\in[\ell-1]\setminus \{{\rm up}_{\De{k}{\mu'}}(y+1)-1,z+d-1,z+d\}$ and
\begin{align*}
\mu_{{\rm up}_{\De{k}{\mu'}}(y+1)-1} =&\  \mu'_{{\rm up}_{\De{k}{\mu'}}(y+1)-1}
\overset{(\ref{eq:finalform-2(1)})}{\geq} \mu'_{{\rm up}_{\De{k}{\mu'}}(y+1)}+1 = \mu_{{\rm up}_{\De{k}{\mu'}}(y+1)},\\
\mu_{z+d-1} =&\ \mu'_{z+d-1} = \lambda_{z+d-1}-1 =\lambda_{z+d}-1= \mu'_{z+d}-1 = \mu_{z+d},\,\text{ and so }\\
\mu_x\geq&\ \mu_{x+1}, \quad \forall x\in[\ell-1]\setminus (z+d);
\end{align*}

\item $\mu_{x} = \lambda_{x}$ for all $x\in[z+d+1,\ell]$.
\end{enumerate}
Plugging~(\mref{eq:finalform-2}) into~(\mref{eq:finalform-ind}) yields the needed result.

For the remainder part of the theorem, if $z = \ell$ or $\lambda_z>\lambda_{z+1}$ with $z\in[\ell-1]$,
then $h=0$ by~(\mref{eq:defnh}) and so $d=0$. Then by the above initial step, we can take $\mu := \lambda-\epsilon_z$.
This completes the proof.
\end{proof}

\begin{remark}
In terms of Theorem~\mref{thm:finalform}, for $\lambda\in\pkl$ and $z\in[\ell]$, $L_z\fg{\lambda}{k}$ can be expressed as a linear summation of some closed $k$-Schur Katalan functions $\fg{\mu}{k}$.
Indeed, the first case in~(\mref{eq:finalform-want}) is exactly as required.
For the second case in~(\mref{eq:finalform-want}),
since $z<\down{\dkl}{z}$ and the maximum vertex in the bounce path containing $z$ is certainly greater than the bottom of the root ideal, the subscript of the lowering operator must be bigger than the bottom of the root ideal after
finitely many steps by repeating Theorem~\mref{thm:finalform}.
\mlabel{re:CSKsum}
\end{remark}

We end this section with an example for better understanding.

\begin{exam}
Let $\lambda:= (5,2,2,1,1,1)\in{\rm P}^5_6$. We express
\begin{equation*}
\fg{\lambda}{k} =
\begin{tikzpicture}[scale=.5,line width=0.8pt,baseline=(a.base)]
\draw (0,2) rectangle (1,3);\node at(0.5,2.5){\scriptsize\( 5 \)};
\filldraw[red,draw=black] (1,2) rectangle (2,3);\node at(1.5,2.5){\scriptsize\( \bullet \)};
\filldraw[red,draw=black] (2,2) rectangle (3,3);\node at(2.5,2.5){\scriptsize\( \bullet \)};
\filldraw[red,draw=black] (3,2) rectangle (4,3);\node at(3.5,2.5){\scriptsize\( \bullet \)};
\filldraw[red,draw=black] (4,2) rectangle (5,3);\node at(4.5,2.5){\scriptsize\( \bullet \)};
\filldraw[red,draw=black] (5,2) rectangle (6,3);\node at(5.5,2.5){\scriptsize\( \bullet \)};
\draw (0,1) rectangle (1,2);
\draw (1,1) rectangle (2,2);
\draw (2,1) rectangle (3,2);
\draw (3,1) rectangle (4,2);
\draw (4,1) rectangle (5,2);
\filldraw[red,draw=black] (5,1) rectangle (6,2);\node at(5.5,1.5){\scriptsize\( \bullet \)};
\draw (0,0) rectangle (1,1);
\draw (1,0) rectangle (2,1);
\draw (2,0) rectangle (3,1);\node at(2.5,0.5){\scriptsize\( 3 \)};
\draw (3,0) rectangle (4,1);
\draw (4,0) rectangle (5,1);
\draw (5,0) rectangle (6,1);
\draw (0,-1) rectangle (1,0);
\draw (1,-1) rectangle (2,0);
\draw (2,-1) rectangle (3,0);
\draw (3,-1) rectangle (4,0);\node at(3.5,-0.5){\scriptsize\( 3 \)};
\draw (4,-1) rectangle (5,0);
\draw (5,-1) rectangle (6,0);
\draw (0,-2) rectangle (1,-1);
\draw (1,-2) rectangle (2,-1);
\draw (2,-2) rectangle (3,-1);
\draw (3,-2) rectangle (4,-1);
\draw (4,-2) rectangle (5,-1);\node at(4.5,-1.5){\scriptsize\( 2 \)};
\draw (5,-2) rectangle (6,-1);
\draw (0,-3) rectangle (1,-2);
\draw (1,-3) rectangle (2,-2);
\draw (2,-3) rectangle (3,-2);
\draw (3,-3) rectangle (4,-2);
\draw (4,-3) rectangle (5,-2);
\draw (5,-3) rectangle (6,-2);
%
%
\draw[green,line width=0.8pt] (1.5,1.5)--(1.5,2);
\draw[green,line width=0.8pt] (1.5,1.5)--(5,1.5);
\draw[green,line width=0.8pt] (5.5,1)--(5.5,-2.5);
\draw[green,line width=0.8pt] (5.5,-2.5)--(6,-2.5);
\node at(5.5,-2.5){\scriptsize\( 1 \)};
\node at(1.5,1.5){\scriptsize\( 4 \)};
\end{tikzpicture}\, ,
\end{equation*}
where the green line represents the bounce path $(1,2,6)$ and row $2$ is the bottom of the root ideal $\dkl$. Then
\begin{equation}
\begin{aligned}
L_2 \fg{\lambda}{k} =&\ L_2 \fg{(5,2,2,1,1,1)}{k}\\
=&\ \fg{(5,1,2,1,1,1)}{k}-L_6\fg{(5,1,2,1,1,1)}{k}+ L_6 \fg{(5,2,2,1,1,1)}{k}\hspace{1cm} (\text{by Lemma~\ref{lem:LamToMu}~\ref{it:LamToMu2}})\\
=&\ \fg{(5,1,1,1,1,1)}{k} - L_6 \fg{(5,1,1,1,1,1)}{k}+ L_6 \fg{(5,2,2,1,1,1)}{k}\\
& \hspace{1.8cm} (\text{by Lemma~\ref{lem:strai}~\ref{it:strai1} for the first and second summands}).
\end{aligned}
\mlabel{eq:ReL3}
\end{equation}
%
%
Further by Theorem~\mref{thm:finalform},
\begin{equation}
L_6 \fg{(5,1,1,1,1,1)}{k} = \fg{(5,1,1,1,1,0)}{k},\quad
L_6 \fg{(5,2,2,1,1,1)}{k} = \fg{(5,2,2,1,1,0)}{k}.
\mlabel{eq:ReL6}
\end{equation}
Substituting~(\mref{eq:ReL6}) into~(\mref{eq:ReL3}) yields
\begin{equation}
L_2 \fg{\lambda}{k} = \fg{(5,1,1,1,1,1)}{k} - \fg{(5,1,1,1,1,0)}{k}+ \fg{(5,2,2,1,1,0)}{k}.
\mlabel{eq:Refinal}
\end{equation}
Hence we conclude that $L_2 \fg{\lambda}{k}$ is a linear summation of three closed $k$-Schur Katalan functions.
\mlabel{ex:Lg}
\end{exam}

\section{Proofs of alternating dual Pieri rule conjecture and $k$-branching conjecture} \mlabel{sec:comproof}
Our goal in this section is to exploit on the proofs of Theorems~\mref{thm:aim1},~\mref{thm:aim2} and~\mref{thm:aim3}.

\subsection{Proof of Theorem~\mref{thm:aim1}} Let us first recall the following result.
\begin{lemma}(\cite[p.~8]{BMS})
Let $\Psi\subseteq\Delta_{\ell}^{+}$ be a root ideal, $M$ a multiset with \textup{supp}$(M)\subseteq [\ell]$, $\gamma\in\mathbb{Z}^{\ell}$ and $d\geq0$. Then
\begin{equation*}		
e_{d}^{\perp}K(\Psi;M;\gamma)=
\sum_{S\subseteq[\ell],|S|=d}K(\Psi;M;\gamma-\epsilon_{S}),
\end{equation*}
where $\epsilon_{S}:=\sum_{i\in S}\epsilon_{i}$.
In particular, $e_{d}^{\perp}K(\Psi;M;\gamma)=0$ for $d>\ell$.
\mlabel{lem:eKatalan}
\end{lemma}\vskip 0.1in

We arrive at a position to prove Theorem~\mref{thm:aim1}.

\begin{proof}[Proof of Theorem~\ref{thm:aim1}]
First,
\begin{equation}
\begin{aligned}
e_d^\perp g_\lambda =&\ e_d^\perp K(\varnothing;\varnothing;\lambda)\hspace{1cm}(\text{by Remark~\ref{re:BMS}~(b)})\\
=&\ \sum_{S\subseteq[\ell],|S|=d}K(\varnothing;\varnothing;\lambda-\epsilon_{S})
\hspace{1cm}(\text{by Lemma~\ref{lem:eKatalan}})\\
=&\ \sum_{S\subseteq[\ell],|S|=d}g_{\lambda-\epsilon_{S}}\hspace{1cm}(\text{by Remark~\ref{re:BMS}~(b)}).
\end{aligned}
\mlabel{eq:aim1-1}
\end{equation}
In particular,
\begin{equation}
e_{d}^{\perp}g_\lambda =e_{d}^{\perp}K(\varnothing;\varnothing;\lambda) =0\, \text{ for }\,d>\ell.
\mlabel{eq:aim1-2}
\end{equation}
Then
\begin{align*}
G_{1^\ell}^{\perp}g_{\lambda}=&\ \Bigg(\sum_{i\geq 0}(-1)^i\binom{\ell-1+i}{\ell-1}e_{\ell+i}\Bigg)^\perp g_{\lambda} \hspace{1cm}(\text{by~(\ref{eq:G})})\\
=&\ \Bigg(\sum_{i\geq 0}(-1)^i\binom{\ell-1+i}{\ell-1}e_{\ell+i}^\perp\Bigg)g_{\lambda}\hspace{1cm}
(\text{by $\perp$ being linear})\\
=&\ e_{\ell}^\perp g_{\lambda}\hspace{1cm}(\text{$e_{\ell+i}^{\perp}g_\lambda=0$ for $i>0$ by (\ref{eq:aim1-2})})\\
=&\ g_{\lambda-1^\ell}\hspace{1cm} (\text{by~(\ref{eq:aim1-1})}).
\end{align*}
This completes the proof.
\end{proof}

\subsection{Proof of Theorem~\mref{thm:aim2}}
Let $\lambda\in\pkl$ and $z\in[\ell]$. For simplicity, abbreviate ${\rm d}_{\lambda}^{a}(z):= \Down{\dkl}{a}{z}$ if $\Down{\dkl}{a}{z}$ is defined.
We give two results as preparation.

\begin{lemma}
Let $\lambda\in\pkl$ and $\bott:=\bott_\lambda$ be the bottom of $\dkl$. Take $z\in[\bott]$ such that $\lambda_z>\lambda_{z+1}$ and ${\rm d}_{\lambda}^{1}(z)<\ell$. Denote $c:=|{\rm path}_{\dkl}(z,{\rm bot}_{\dkl}(z))|$.
For each $a\in[c-1]$, if $\lambda_{{\rm d}_{\lambda}^{a}(z)} = \lambda_{{\rm d}_{\lambda}^{a}(z)+1}$, then
\begin{equation}
L_{{\rm d}_{\lambda}^{a}(z)}\fg{\lambda}{k} - L_{{\rm d}_{\lambda}^{a}(z)}\fg{\lambda-\epsilon_z}{k}=
\left\{
\begin{array}{ll}
 0, & \quad \text{if $a=c-1$};\\
 L_{{\rm d}_{\lambda}^{a+1}(z)}\fg{\lambda}{k} - L_{{\rm d}_{\lambda}^{a+1}(z)}\fg{\lambda-\epsilon_z}{k}, & \quad \text{if $a\in[c-2]$}.
\end{array}
\right.
\mlabel{eq:indkey}
\end{equation}
\mlabel{lem:indkey}
\end{lemma}

\begin{proof}
Since $\lambda_z>\lambda_{z+1}$ and ${\rm d}_{\lambda}^{1}(z)<\ell$, it follows from Remark~\mref{re:factroot}~(b) that there is a ceiling in columns ${\rm d}_{\lambda}^{1}(z),{\rm d}_{\lambda}^{1}(z)+1$ in the root ideal $\dkl$. Then
\begin{equation}
{\rm top}_{\De{k}{\lambda-\epsilon_{{\rm d}_{\lambda}^{a}(z)}}}({\rm d}_{\lambda}^{a}(z)+1)-1 > {\rm top}_{\De{k}{\lambda-\epsilon_{{\rm d}_{\lambda}^{a}(z)}}}({\rm d}_{\lambda}^{a}(z)).
\mlabel{eq:T1}
\end{equation}
Further, the hypothesises $\lambda_z>\lambda_{z+1}$ and ${\rm d}_{\lambda}^{1}(z)<\ell$ show that $(z,{\rm d}_{\lambda}^{1}(z)+1)$ is removable in $\De{k}{\lambda-\epsilon_z}$.
Hence
\begin{equation}
y:={\rm top}_{\De{k}{\lambda-\epsilon_z-\epsilon_{{\rm d}_{\lambda}^{a}(z)}}}({\rm d}_{\lambda}^{a}(z))>{\rm top}_{\De{k}{\lambda-\epsilon_z-\epsilon_{{\rm d}_{\lambda}^{a}(z)}}}({\rm d}_{\lambda}^{a}(z)+1)\,\text{ and }\,\uup{\De{k}{\lambda-\epsilon_z-\epsilon_{{\rm d}_{\lambda}^{a}(z)}}}{y+1} =z.
\mlabel{eq:T2}
\end{equation}
By $\lambda\in\pkl$ and $\lambda_{{\rm d}_{\lambda}^{a}(z)} = \lambda_{{\rm d}_{\lambda}^{a}(z)+1}$,
\begin{equation}
\begin{aligned}
(\lambda-\epsilon_{{\rm d}_{\lambda}^{a}(z)})_{{\rm d}_{\lambda}^{a}(z)} =&\ \lambda_{{\rm d}_{\lambda}^{a}(z)}-1 =\lambda_{{\rm d}_{\lambda}^{a}(z)+1}-1 = (\lambda-\epsilon_{{\rm d}_{\lambda}^{a}(z)})_{{\rm d}_{\lambda}^{a}(z)+1}-1,\\
(\lambda-\epsilon_{{\rm d}_{\lambda}^{a}(z)})_x\geq&\ (\lambda-\epsilon_{{\rm d}_{\lambda}^{a}(z)})_{x+1}, \quad \forall x\in[\ell-1]\setminus {\rm d}_{\lambda}^{a}(z).
\end{aligned}
\mlabel{eq:C1}
\end{equation}
Notice that
\begin{align*}
\lambda-\epsilon_z\in\pkl \,\text{ by }\, \lambda_z>\lambda_{z+1},\quad
(\lambda-\epsilon_z)_{{\rm d}_{\lambda}^{a}(z)}=\lambda_{{\rm d}_{\lambda}^{a}(z)} = \lambda_{{\rm d}_{\lambda}^{a}(z)+1} = (\lambda-\epsilon_z)_{{\rm d}_{\lambda}^{a}(z)+1}.
\end{align*}
Then
\begin{equation}
\begin{aligned}
 (\lambda-\epsilon_z-\epsilon_{{\rm d}_{\lambda}^{a}(z)})_{{\rm d}_{\lambda}^{a}(z)} =&\ \lambda_{{\rm d}_{\lambda}^{a}(z)}-1 = \lambda_{{\rm d}_{\lambda}^{a}(z)+1}-1=(\lambda-\epsilon_z-\epsilon_{{\rm d}_{\lambda}^{a}(z)})_{{\rm d}_{\lambda}^{a}(z)+1}-1,\\
(\lambda-\epsilon_z-\epsilon_{{\rm d}_{\lambda}^{a}(z)})_x\geq&\ (\lambda-\epsilon_z-\epsilon_{{\rm d}_{\lambda}^{a}(z)})_{x+1},\quad \forall x\in[\ell-1]\setminus {\rm d}_{\lambda}^{a}(z).
\end{aligned}
\mlabel{eq:C2}
\end{equation}
According to $a\in [c-1]$, we can divide the remaining proof into the following two cases.

{\bf Case~1.} $a=c-1$. Since $c:=|{\rm path}_{\dkl}(z,{\rm bot}_{\dkl}(z))|$,
$${\rm d}_{\lambda}^{a}(z) = {\rm bot}_{\dkl}(z) \in[\bott+1,\ell-1].$$
By Lemma~\mref{lem:LamToMu}~\mref{it:LamToMu1},
\begin{equation}
L_{{\rm d}_{\lambda}^{a}(z)}\fg{\lambda}{k} = \fg{\lambda-\epsilon_{{\rm d}_{\lambda}^{a}(z)}}{k},\quad  L_{{\rm d}_{\lambda}^{a}(z)}\fg{\lambda-\epsilon_z}{k} = \fg{\lambda-\epsilon_z-\epsilon_{{\rm d}_{\lambda}^{a}(z)}}{k}.
\mlabel{eq:lem420}
\end{equation}
From~(\ref{eq:T1}),~(\ref{eq:C1}) and Lemma~\ref{lem:strai}~\ref{it:strai1},
\begin{equation}
\fg{\lambda-\epsilon_{{\rm d}_{\lambda}^{a}(z)}}{k} = \fg{\lambda-\epsilon_{{\rm d}_{\lambda}^{a}(z)}-\epsilon_{{\rm d}_{\lambda}^{a}(z)+1}}{k}.
\mlabel{eq:lem421}
\end{equation}
In terms of (\ref{eq:T2}),~(\ref{eq:C2}) and Lemma~\ref{lem:strai}~\ref{it:strai2},
\begin{equation}
\fg{\lambda-\epsilon_z-\epsilon_{{\rm d}_{\lambda}^{a}(z)}}{k} = \fg{(\lambda-\epsilon_z-\epsilon_{{\rm d}_{\lambda}^{a}(z)})+\epsilon_{z}-\epsilon_{{\rm d}_{\lambda}^{a}(z)+1}}{k} \\ =\fg{\lambda-\epsilon_{{\rm d}_{\lambda}^{a}(z)}-\epsilon_{{\rm d}_{\lambda}^{a}(z)+1}}{k}.
\mlabel{eq:lem422}
\end{equation}
Substituting~(\mref{eq:lem421}) and~(\mref{eq:lem422}) into~(\mref{eq:lem420}) and subtracting the two equations in~(\mref{eq:lem420}), we conclude
\begin{equation*}
L_{{\rm d}_{\lambda}^{a}(z)}\fg{\lambda}{k} - L_{{\rm d}_{\lambda}^{a}(z)}\fg{\lambda-\epsilon_z}{k}= \fg{\lambda-\epsilon_{{\rm d}_{\lambda}^{a}(z)}}{k}- \fg{\lambda-\epsilon_z-\epsilon_{{\rm d}_{\lambda}^{a}(z)}}{k} = \fg{\lambda-\epsilon_{{\rm d}_{\lambda}^{a}(z)}-\epsilon_{{\rm d}_{\lambda}^{a}(z)+1}}{k}-\fg{\lambda-\epsilon_{{\rm d}_{\lambda}^{a}(z)}-\epsilon_{{\rm d}_{\lambda}^{a}(z)+1}}{k}=0.
\end{equation*}

{\bf Case~2.}  $a\in[c-2]$. Then ${\rm d}_{\lambda}^{a}(z)\in[\bott]$ and
\begin{equation}
\begin{aligned}
L_{{\rm d}_{\lambda}^{a}(z)}\fg{\lambda}{k} =&\ \fg{\lambda-\epsilon_{{\rm d}_{\lambda}^{a}(z)}}{k}-L_{\down{\dkl}{{\rm d}_{\lambda}^{a}(z)}}\fg{\lambda-\epsilon_{{\rm d}_{\lambda}^{a}(z)}}{k}+ L_{\down{\dkl}{{\rm d}_{\lambda}^{a}(z)}}\fg{\lambda}{k}\hspace{1cm}(\text{by Lemma~\ref{lem:LamToMu}~\ref{it:LamToMu2}})\\
=&\ \fg{\lambda-\epsilon_{{\rm d}_{\lambda}^{a}(z)}}{k}-L_{{\rm d}_{\lambda}^{a+1}(z)}\fg{\lambda-\epsilon_{{\rm d}_{\lambda}^{a}(z)}}{k}+ L_{{\rm d}_{\lambda}^{a+1}(z)}\fg{\lambda}{k}\hspace{1cm}(\text{by $\down{\dkl}{{\rm d}_{\lambda}^{a}(z)} = {\rm d}_{\lambda}^{a+1}(z)$})\\
=&\ \fg{\lambda-\epsilon_{{\rm d}_{\lambda}^{a}(z)}-\epsilon_{{\rm d}_{\lambda}^{a}(z)+1}}{k}-L_{{\rm d}_{\lambda}^{a+1}(z)}\fg{\lambda-\epsilon_{{\rm d}_{\lambda}^{a}(z)}-\epsilon_{{\rm d}_{\lambda}^{a}(z)+1}}{k}+ L_{{\rm d}_{\lambda}^{a+1}(z)}\fg{\lambda}{k}\\
& \hspace{3cm}(\text{$\fg{\lambda-\epsilon_{{\rm d}_{\lambda}^{a}(z)}}{k} = \fg{\lambda-\epsilon_{{\rm d}_{\lambda}^{a}(z)}-\epsilon_{{\rm d}_{\lambda}^{a}(z)+1}}{k}$ by~(\ref{eq:T1}),~(\ref{eq:C1}) and Lemma~\ref{lem:strai}~\ref{it:strai1}}).
\end{aligned}
\mlabel{eq:indkey1}
\end{equation}
Since $(z,{\rm d}_{\lambda}^{1}(z))$ is removable in $\De{k}{\lambda}$, we have that $(z,{\rm d}_{\lambda}^{1}(z)+1)$ is removable in $\De{k}{\lambda-\epsilon_z}$, and so ${\rm d}_{\lambda}^{a}(z)\in[b']$, where $b'$ is the bottom of the $\De{k}{\lambda-\epsilon_z}$. Thus
\begin{align}
L_{{\rm d}_{\lambda}^{a}(z)}\fg{\lambda-\epsilon_z}{k} =&\ \fg{\lambda-\epsilon_z-\epsilon_{{\rm d}_{\lambda}^{a}(z)}}{k}-L_{\down{\De{k}{\lambda-\epsilon_z}}{{\rm d}_{\lambda}^{a}(z)}}\fg{\lambda-\epsilon_z-\epsilon_{{\rm d}_{\lambda}^{a}(z)}}{k}+ L_{\down{\De{k}{\lambda-\epsilon_z}}{{\rm d}_{\lambda}^{a}(z)}}\fg{\lambda-\epsilon_z}{k}\hspace{1cm}(\text{by Lemma~\ref{lem:LamToMu}~\ref{it:LamToMu2}})\notag\\
=&\ \fg{\lambda-\epsilon_z-\epsilon_{{\rm d}_{\lambda}^{a}(z)}}{k}-L_{{\rm d}_{\lambda}^{a+1}(z)}\fg{\lambda-\epsilon_z-\epsilon_{{\rm d}_{\lambda}^{a}(z)}}{k}+ L_{{\rm d}_{\lambda}^{a+1}(z)}\fg{\lambda-\epsilon_z}{k}\hspace{1cm}(\text{by $\down{\De{k}{\lambda-\epsilon_z}}{{\rm d}_{\lambda}^{a}(z)} = {\rm d}_{\lambda}^{a+1}(z)$})\notag\\
=&\ \fg{(\lambda-\epsilon_z-\epsilon_{{\rm d}_{\lambda}^{a}(z)})+\epsilon_z-\epsilon_{{\rm d}_{\lambda}^{a}(z)+1}}{k}-L_{{\rm d}_{\lambda}^{a+1}(z)}\fg{(\lambda-\epsilon_z-\epsilon_{{\rm d}_{\lambda}^{a}(z)})+\epsilon_z-\epsilon_{{\rm d}_{\lambda}^{a}(z)+1}}{k}+ L_{{\rm d}_{\lambda}^{a+1}(z)}\fg{\lambda-\epsilon_z}{k}\notag\\
& \hspace{3cm}(\text{$\fg{\lambda-\epsilon_z-\epsilon_{{\rm d}_{\lambda}^{a}(z)}}{k}=\fg{(\lambda-\epsilon_z-\epsilon_{{\rm d}_{\lambda}^{a}(z)})+\epsilon_z-\epsilon_{{\rm d}_{\lambda}^{a}(z)+1}}{k}$ by~(\ref{eq:T2}),~(\ref{eq:C2}) and Lemma~\ref{lem:strai}~\ref{it:strai2}})\notag\\
=&\ \fg{\lambda-\epsilon_{{\rm d}_{\lambda}^{a}(z)}-\epsilon_{{\rm d}_{\lambda}^{a}(z)+1}}{k}-L_{{\rm d}_{\lambda}^{a+1}(z)}\fg{\lambda-\epsilon_{{\rm d}_{\lambda}^{a}(z)}-\epsilon_{{\rm d}_{\lambda}^{a}(z)+1}}{k}+ L_{{\rm d}_{\lambda}^{a+1}(z)}\fg{\lambda-\epsilon_z}{k}.\mlabel{eq:indkey2}
\end{align}
Finally, (\ref{eq:indkey1}) subtracting~(\ref{eq:indkey2}) yields
\begin{align*}
L_{{\rm d}_{\lambda}^{a}(z)}\fg{\lambda}{k}- L_{{\rm d}_{\lambda}^{a}(z)}\fg{\lambda-\epsilon_z}{k} =&\ \fg{\lambda-\epsilon_{{\rm d}_{\lambda}^{a}(z)}-\epsilon_{{\rm d}_{\lambda}^{a}(z)+1}}{k}-L_{{\rm d}_{\lambda}^{a+1}(z)}\fg{\lambda-\epsilon_{{\rm d}_{\lambda}^{a}(z)}-\epsilon_{{\rm d}_{\lambda}^{a}(z)+1}}{k}+ L_{{\rm d}_{\lambda}^{a+1}(z)}\fg{\lambda}{k}\\
&\ -\big( \fg{\lambda-\epsilon_{{\rm d}_{\lambda}^{a}(z)}-\epsilon_{{\rm d}_{\lambda}^{a}(z)+1}}{k}-L_{{\rm d}_{\lambda}^{a+1}(z)}\fg{\lambda-\epsilon_{{\rm d}_{\lambda}^{a}(z)}-\epsilon_{{\rm d}_{\lambda}^{a}(z)+1}}{k}+ L_{{\rm d}_{\lambda}^{a+1}(z)}\fg{\lambda-\epsilon_z}{k} \big)\\
%
%
=&\ L_{{\rm d}_{\lambda}^{a+1}(z)}\fg{\lambda}{k}-L_{{\rm d}_{\lambda}^{a+1}(z)}\fg{\lambda-\epsilon_z}{k}.
\end{align*}
This completes the proof.
\end{proof}

\begin{prop}
Let $\lambda\in\pkl$ and $z\in[\ell]$ with $\lambda_z > \lambda_{z+1}$. Then
\begin{equation}
L_z \fg{\lambda}{k} = \sum_{\substack{\mu\in\pkl,|\mu|<|\lambda| \\ \mu_x = \lambda_x\,\textup{for}\,x\in[z-1]}}a_{\lambda\mu}\fg{\mu}{k},
\mlabel{eq:pp43}
\end{equation}
where $a_{\lambda\mu} = 0$ or $a_{\lambda\mu} = (-1)^{|\lambda|-|\mu|-1}$.
\mlabel{prop:ind}
\end{prop}

\begin{proof}
If $z\in[\bott+1,\ell]$, it follows from Theorem~\mref{thm:finalform} and $\lambda_z > \lambda_{z+1}$ that $L_z \fg{\lambda}{k} = \fg{\mu}{k}$ with $\mu := \lambda-\epsilon_z\in\pkl$. Since $$(-1)^{|\lambda|-|\mu|-1} = (-1)^{|\lambda|-|\lambda-\epsilon_z|-1} = 1,$$
the result is valid.
Suppose $z\in[\bott]$. Theorem~\mref{thm:finalform} implies that
\begin{equation}
L_z\fg{\lambda}{k} = \fg{\mu}{k}-
L_{{\rm d}_{\lambda}^{1}(z)}\fg{\mu}{k}+
L_{{\rm d}_{\lambda}^{1}(z)}\fg{\lambda}{k},
\mlabel{eq:ind1}
\end{equation}
where $\mu = \lambda-\epsilon_z\in\pkl$ by $\lambda_z > \lambda_{z+1}$. We break the remaining proof into the following two cases.

{\bf Case~1.} ${\rm d}_{\lambda}^{1}(z)=\ell$. Then
\begin{align*}
L_z\fg{\lambda}{k} =&\ \fg{\mu}{k}-
L_{{\rm d}_{\lambda}^{1}(z)}\fg{\mu}{k}+
L_{{\rm d}_{\lambda}^{1}(z)}\fg{\lambda}{k}\hspace{1cm}(\text{by~(\ref{eq:ind1})})\\
=&\ \fg{\mu}{k}-
L_{\ell}\fg{\mu}{k}+
L_{\ell}\fg{\lambda}{k}\\
=&\ \fg{\mu}{k}-\fg{\mu-\epsilon_\ell}{k}+\fg{\lambda-\epsilon_\ell}{k}\hspace{1cm}
(\text{by Theorem~\ref{thm:finalform} for the second and third summands}),
\end{align*}
which is the required form in~(\mref{eq:pp43}), as $\mu-\epsilon_\ell,\lambda-\epsilon_\ell\in\pkl$ and
\begin{align*}
a_{\lambda\mu}=&\ (-1)^{|\lambda|-|\mu|-1}=(-1)^{|\lambda|-|\lambda-\epsilon_z|-1}=1,\\
a_{\lambda,\mu-\epsilon_\ell}=&\ (-1)^{|\lambda|-|\mu-\epsilon_\ell|-1}
=(-1)^{|\lambda|-|\lambda-\epsilon_z-\epsilon_\ell|-1}
=-1,\\
a_{\lambda,\lambda-\epsilon_\ell}
=&\ (-1)^{|\lambda|-|\lambda-\epsilon_\ell|-1}
=1.
\end{align*}

{\bf Case~2.} ${\rm d}_{\lambda}^{1}(z)<\ell$. Denote
$$c:=|{\rm path}_{\dkl}(z,{\rm bot}_{\dkl}(z))|.$$
Notice that $\lambda_{{\rm d}_{\lambda}^{a}(z)}\geq\lambda_{{\rm d}_{\lambda}^{a}(z)+1}$ for $a\in[c-1]$ by $\lambda\in\pkl$. There are two subcases to consider.

{\bf Subcase~2.1.}
$\lambda_{{\rm d}_{\lambda}^{a}(z)}=\lambda_{{\rm d}_{\lambda}^{a}(z)+1}$ for all $a\in[c-1]$. By Lemma~\mref{lem:indkey},
\begin{equation}
L_{{\rm d}_{\lambda}^{1}(z)}\fg{\lambda}{k}
-
L_{{\rm d}_{\lambda}^{1}(z)}\fg{\mu}{k} = L_{{\rm d}_{\lambda}^{2}(z)}\fg{\lambda}{k}
-
L_{{\rm d}_{\lambda}^{2}(z)}\fg{\mu}{k}=\cdots=L_{{\rm d}_{\lambda}^{c-1}(z)}\fg{\lambda}{k}
-
L_{{\rm d}_{\lambda}^{c-1}(z)}\fg{\mu}{k}=0,
\mlabel{eq:LL0}
\end{equation}
where the last equal is due to the first case in (\mref{eq:indkey}) and the other equals are because of the second case in (\mref{eq:indkey}).
Hence
$$L_z\fg{\lambda}{k} \overset{(\ref{eq:ind1})}{=} \fg{\mu}{k}-
L_{{\rm d}_{\lambda}^{1}(z)}\fg{\mu}{k}+
L_{{\rm d}_{\lambda}^{1}(z)}\fg{\lambda}{k} \overset{(\ref{eq:LL0})}{=}\fg{\mu}{k},$$
which is the required form in~(\mref{eq:pp43}), as $$(-1)^{|\lambda|-|\mu|-1} = (-1)^{|\lambda|-|\lambda-\epsilon_z|-1} = 1.$$

{\bf Subcase~2.2.} $\lambda_{{\rm d}_{\lambda}^{a}(z)}\neq\lambda_{{\rm d}_{\lambda}^{a}(z)+1}$ for some $a\in[c-1]$. Take $e\in[c-1]$ to be the smallest number such that $\lambda_{{\rm d}_{\lambda}^{e}(z)}>\lambda_{{\rm d}_{\lambda}^{e}(z)+1}$.
We proceed the induction on $\ell-z$ to prove~(\mref{eq:pp43}).
For the initial step of $\ell-z$, $z$ is the maximum vertex of the bounce path passing through $z$, which may be any number in $[\bott+1,\ell]$, that is, $\ell-z$ may be any number in $[0, \ell-\bott-1]$.
So the initial step follows from
the very beginning of the proof.

Consider the inductive step of $\ell - z$. By the definition of $e$,
$$\lambda_{{\rm d}_{\lambda}^{a}(z)} = \lambda_{{\rm d}_{\lambda}^{a}(z)+1}\,\quad \forall a\in [e-1].$$
So we can use the second case of (\ref{eq:indkey}) to get
\begin{equation}
L_{{\rm d}_{\lambda}^{1}(z)}\fg{\lambda}{k}
-
L_{{\rm d}_{\lambda}^{1}(z)}\fg{\mu}{k} = L_{{\rm d}_{\lambda}^{2}(z)}\fg{\lambda}{k}
-
L_{{\rm d}_{\lambda}^{2}(z)}\fg{\mu}{k}=\cdots=L_{{\rm d}_{\lambda}^{e}(z)}\fg{\lambda}{k}
-
L_{{\rm d}_{\lambda}^{e}(z)}\fg{\mu}{k}.
\mlabel{eq:pp431}
\end{equation}
Plugging~(\mref{eq:pp431}) into~(\mref{eq:ind1}),
\begin{align}
L_z\fg{\lambda}{k} =&\ \fg{\mu}{k}-L_{{\rm d}_{\lambda}^{e}(z)}\fg{\mu}{k}+ L_{{\rm d}_{\lambda}^{e}(z)}\fg{\lambda}{k}.
\mlabel{eq:lzgl}
\end{align}
Notice that $\mu\in\pkl$ and
$$\mu_{{\rm d}_{\lambda}^{e}(z)} = (\lambda-\epsilon_z)_{{\rm d}_{\lambda}^{e}(z)} = \lambda_{{\rm d}_{\lambda}^{e}(z)} > \lambda_{{\rm d}_{\lambda}^{e}(z)+1}=(\lambda-\epsilon_z)_{{\rm d}_{\lambda}^{e}(z)+1} =\mu_{{\rm d}_{\lambda}^{e}(z)+1},$$
showing that the hypothesis conditions of the result are satisfied for $\mu$ and ${\rm d}_{\lambda}^{e}(z)$.
Further by the definition of ${\rm d}_{\lambda}^{e}(z)$, we have ${\rm d}_{\lambda}^{e}(z)>z$,
that is, $\ell - {\rm d}_{\lambda}^{e}(z) < \ell -z$.
For the second and third summands in~(\mref{eq:lzgl}), using the inductive hypothesis yields
\begin{equation}
\begin{aligned}
L_{{\rm d}_{\lambda}^{e}(z)}\fg{\mu}{k} =&\ \sum_{\substack{\mu^{(1)}\in\pkl,|\mu^{(1)}|<|\mu| \\ \mu^{(1)}_x =\mu_x\,\text{for}\,x\in[{\rm d}_{\lambda}^{e}(z)-1]}}a_{\mu\mu^{(1)}}\fg{\mu^{(1)}}{k}\,\text{ with }\,  a_{\mu\mu^{(1)}}=0\,\text{ or }\, a_{\mu\mu^{(1)}} = (-1)^{|\mu|-|\mu^{(1)}|-1},\\
L_{{\rm d}_{\lambda}^{e}(z)}\fg{\lambda}{k}=&\ \sum_{\substack{\mu^{(2)}\in\pkl,|\mu^{(2)}|<|\lambda| \\ \mu^{(2)}_x = \lambda_x\,\text{for}\,x\in[{\rm d}_{\lambda}^{e}(z)-1]}}a_{\lambda\mu^{(2)}}\fg{\mu^{(2)}}{k}\,\text{ with }\, a_{\lambda\mu^{(2)}}= 0\,\text{ or }\, a_{\lambda\mu^{(2)}} = (-1)^{|\lambda|-|\mu^{(2)}|-1}.
\mlabel{eq:indmu}
\end{aligned}
\end{equation}
Inserting~(\ref{eq:indmu}) into~(\mref{eq:lzgl}),
\begin{equation}
L_z\fg{\lambda}{k} = \fg{\mu}{k}+\sum_{\substack{\mu^{(1)}\in\pkl,|\mu^{(1)}|<|\mu| \\ \mu^{(1)}_x =\mu_x\,\text{for}\,x\in[{\rm d}_{\lambda}^{e}(z)-1]}}(-1)a_{\mu\mu^{(1)}}\fg{\mu^{(1)}}{k}
+\sum_{\substack{\mu^{(2)}\in\pkl,|\mu^{(2)}|<|\lambda| \\ \mu^{(2)}_x = \lambda_x\,\text{for}\,x\in[{\rm d}_{\lambda}^{e}(z)-1]}}a_{\lambda\mu^{(2)}}\fg{\mu^{(2)}}{k}.
\mlabel{eq:ind3}
\end{equation}
Finally, since
\begin{align*}
1=&\ (-1)^{|\lambda|-|\mu|-1},\,\text{ and}\\
(-1)a_{\mu\mu^{(1)}}=&\ 0\,\text{ or }\,(-1)a_{\mu\mu^{(1)}} = (-1)(-1)^{|\mu|-|\mu^{(1)}|-1} = (-1)^{|\lambda|-|\mu^{(1)}|-1},
\end{align*}
(\mref{eq:ind3}) is the required form of~(\mref{eq:pp43}). This completes the proof.
\end{proof}

Now we are ready for the proof of Theorem~\ref{thm:aim2}.

\begin{proof}[Proof of Theorem~\ref{thm:aim2}]
First,
\begin{align}
G_{1^m}^{\perp}\fg{\lambda}{k} =&\ \Bigg( \sum_{i\geq 0}(-1)^i\binom{m-1+i}{m-1}e_{m+i} \Bigg)^\perp \fg{\lambda}{k}\hspace{1cm}(\text{by~(\ref{eq:G})})\notag\\
=&\ \sum_{i\geq 0}(-1)^i\binom{m-1+i}{m-1}e_{m+i}^\perp K(\dkl;\dkl;\lambda)\hspace{1cm}(\text{by Definition~\ref{defn:cKataFunc}})\notag\\
=&\ \sum^{\ell-m}_{i\geq 0}(-1)^i\binom{m-1+i}{m-1} \Bigg( \sum_{S\subseteq[\ell],|S|=m+i}K(\dkl;\dkl;\lambda-\epsilon_{S})\Bigg)\hspace{1cm}
(\text{by Lemma~\ref{lem:eKatalan}})\notag\\
=&\ \sum^{\ell-m}_{i\geq 0}(-1)^i\binom{m-1+i}{m-1}\Bigg( \sum_{S\subseteq[\ell],|S|=m+i}\prod_{z\in S}L_z K(\dkl;\dkl;\lambda) \Bigg)\hspace{1cm}(\text{by Remark~\ref{re:BMS}~(a)})\notag\\
=&\ \sum^{\ell-m}_{i\geq 0}(-1)^i\binom{m-1+i}{m-1}\Bigg( \sum_{S\subseteq[\ell],|S|=m+i}\prod_{z\in S}L_z \fg{\lambda}{k}\Bigg)\hspace{1cm}(\text{by Definition~\ref{defn:cKataFunc}}).\mlabel{eq:aim2-1}
\end{align}
Claim: for $\lambda\in\pkl$ and $S\subseteq[\ell]$ with $|S|\geq 1$, if $\lambda_z>\lambda_{z+1}$ for all $z\in S$, then
\begin{equation}
\prod_{z\in S}L_z \fg{\lambda}{k} = \sum_{\mu\in\pkl,|\mu|<|\lambda|}a_{\lambda\mu}\fg{\mu}{k},
\mlabel{eq:aim2-ind}
\end{equation}
where $a_{\lambda\mu}=0$ or $a_{\lambda\mu} = (-1)^{|\lambda|-|\mu|-|S|}$. Substitute~(\mref{eq:aim2-ind}) into~(\mref{eq:aim2-1}),
\begin{align*}
G_{1^m}^{\perp}\fg{\lambda}{k}
= \sum^{\ell-m}_{i\geq 0}\sum_{S\subseteq[\ell],|S|=m+i}  \sum_{\mu\in\pkl,|\mu|<|\lambda|}
(-1)^i\binom{m-1+i}{m-1}a_{\lambda\mu}\fg{\mu}{k},
\end{align*}
where $(-1)^i\binom{m-1+i}{m-1}a_{\lambda\mu}=0$ or
\[
(-1)^i\binom{m-1+i}{m-1}a_{\lambda\mu} = (-1)^i\binom{m-1+i}{m-1}(-1)^{|\lambda|-|\mu|-(m+i)} = \binom{m-1+i}{m-1}(-1)^{|\lambda|-|\mu|-m}.
\]
Taking $c_{\lambda\mu} := (-1)^i\binom{m-1+i}{m-1}a_{\lambda\mu}$, then $(-1)^{|\lambda|-|\mu|-m}c_{\lambda\mu} = 0 \in\ZZ_{\geq0}$ or
\[
(-1)^{|\lambda|-|\mu|-m}c_{\lambda\mu} = (-1)^{|\lambda|-|\mu|-m} \binom{m-1+i}{m-1}(-1)^{|\lambda|-|\mu|-m} = \binom{m-1+i}{m-1} \in\ZZ_{\geq0},
\]
as required.

The remaining proof of the claim proceeds by induction on $|S|\geq 1$.
For the initial step of $|S|=1$, since $\lambda_z>\lambda_{z+1}$ for all $z\in S$ by the hypothesis, we have
\begin{align*}
\prod_{z\in S}L_z \fg{\lambda}{k} =&\ L_z\fg{\lambda}{k}\hspace{1cm}(\text{by assuming $S=\{z\}$})\\
=&\ \sum_{\substack{\mu\in\pkl,|\mu|<|\lambda| \\ \mu_x = \lambda_x\,\textup{for}\,x\in[z-1]}}a_{\lambda\mu}\fg{\mu}{k}\hspace{1cm}
(\text{by $\lambda_z>\lambda_{z+1}$ and Proposition~\ref{prop:ind}})\\
=&\ \sum_{\mu\in\pkl,|\mu|<|\lambda|}a_{\lambda\mu}\fg{\mu}{k}\hspace{1cm}(\text{by taking $a_{\lambda\mu}:=0$ if $\mu_x \neq\lambda_x$ for some $x\in[z-1]$}),
\end{align*}
where $a_{\lambda\mu}=0$ or $a_{\lambda\mu} = (-1)^{|\lambda|-|\mu|-|S|}$.
Consider the inductive step of $|S| > 1$. Take $x$ to be the maximum element in $S$.
Then
\begin{align}
\prod_{z\in S}L_z \fg{\lambda}{k} =
&\ \prod_{z\in S\setminus x}L_z \Big(L_x \fg{\lambda}{k}\Big)\notag\\
=&\ \prod_{z\in S\setminus x}L_z\Bigg( \sum_{\substack{\mu'\in\pkl,|\mu'|<|\lambda| \\ \mu'_y = \lambda_y\,\textup{for}\,y\in[x-1]}}a_{\lambda\mu'}\fg{\mu'}{k}
\Bigg)\hspace{1cm}(\text{by $\lambda_x>\lambda_{x+1}$ and Proposition~\ref{prop:ind}})\notag\\
=&\ \sum_{\substack{\mu'\in\pkl,|\mu'|<|\lambda| \\ \mu'_y = \lambda_y\,\textup{for}\,y\in[x-1]}}a_{\lambda\mu'}\prod_{z\in S\setminus x}L_z\fg{\mu'}{k}
\hspace{1cm}(\text{by $L_z$ being additive}),\mlabel{eq:LgtoaLg}
\end{align}
where $a_{\lambda\mu'}= 0$ or $a_{\lambda\mu'} = (-1)^{|\lambda|-|\mu'|-1}$.
By the definition of $x$, for each $z\in S\setminus x$ and each $\mu'$ in~(\mref{eq:LgtoaLg}), we have $z\in[x-1]$ and so
\[
\mu'_z = \lambda_z >\lambda_{z+1} = \mu'_{z+1},
\]
implying that the hypothesis condition of the claim is satisfied for $S\setminus x$ and $\mu'$. Then by the inductive hypothesis,
for each $\mu'$ in~(\mref{eq:LgtoaLg}),
\begin{equation}
\prod_{z\in S\setminus x}L_z\fg{\mu'}{k} = \sum_{\mu\in\pkl,|\mu|<|\mu'|}a_{\mu'\mu}\fg{\mu}{k},
\mlabel{eq:Lgind}
\end{equation}
where $a_{\mu'\mu}= 0$ or $a_{\mu'\mu}= (-1)^{|\mu'|-|\mu|-|S\setminus x|}$.
Substituting~(\mref{eq:Lgind}) into~(\mref{eq:LgtoaLg}),
\begin{align*}
\prod_{z\in S}L_z \fg{\lambda}{k}= \sum_{\mu\in\pkl,|\mu|<|\lambda|}a_{\lambda\mu'}a_{\mu'\mu}\fg{\mu}{k}.
\end{align*}
Hence~(\mref{eq:aim2-ind}) is valid by
$a_{\lambda\mu'}a_{\mu'\mu} = 0$ or
$$a_{\lambda\mu'}a_{\mu'\mu}= (-1)^{|\lambda|-|\mu'|-1}(-1)^{|\mu'|-|\mu|-|S\setminus x|} = (-1)^{|\lambda|-|\mu|-|S|}.$$
This completes the proof of the claim and the proof of the result.
\end{proof}

\subsection{Proof of Theorem~\mref{thm:aim3}}
We finally arrive at the proof of Theorem~\mref{thm:aim3}.

\begin{proof}[Proof of Theorem~\ref{thm:aim3}]
By \cite[Proposition~2.16~(c)]{BMS}, closed $k$-Schur Katalan functions $\fg{\lambda}{k}$ with $\lambda\in\pkl$ satisfy shift invariance
\begin{equation}
G_{1^\ell}^\perp \fg{\lambda+1^\ell}{k+1} = \fg{\lambda}{k}.
\mlabel{eq:shiftinvariance}
\end{equation}
Suppose $\lambda\in\spkl$. Then
$$(\lambda+1^\ell)_x = \lambda_x+1 > \lambda_{x+1}+1 = (\lambda+1^\ell)_{x+1}, \quad \forall x\in[\ell-1],$$
and so $\lambda+1^\ell\in{\rm sP}_\ell^{k+1}$. Hence
\begin{align*}
\fg{\lambda}{k} =&\ G_{1^\ell}^\perp \fg{\lambda+1^\ell}{k+1}\hspace{1cm}
(\text{by~(\ref{eq:shiftinvariance})})\\
=&\ \sum_{\mu\in{\rm P}}c_{\lambda+1^\ell,\mu}\fg{\mu}{k+1}\hspace{1cm}
(\text{by Theorem~\ref{thm:aim2} and taking $m=\ell$}),
\end{align*}
where $(-1)^{|\lambda+1^\ell|-|\mu|-\ell}c_{\lambda+1^\ell,\mu} \in\ZZ_{\geq0}.$
Taking $a_{\lambda\mu} := c_{\lambda+1^\ell,\mu}$, we obtain
$$
(-1)^{|\lambda|-|\mu|}a_{\lambda\mu} = (-1)^{|\lambda|-|\mu|}c_{\lambda+1^\ell,\mu}=(-1)^{|\lambda+1^\ell|-|\mu|-\ell}c_{\lambda+1^\ell,\mu} \in\ZZ_{\geq0}.
$$
This completes the proof.
\end{proof}

We end the paper with the following remark.

\begin{remark}
\begin{enumerate}
\item The condition that $\lambda_z>\lambda_{z+1}$ is necessary in Proposition~\mref{prop:ind}.
A counterexample is Example~\mref{ex:Lg}. In more details, in that example,
\begin{align*}
\lambda= (5,2,2,1,1,1), \quad \lambda_2=2=\lambda_{3}, \quad L_2 \fg{(5,2,2,1,1,1)}{k} = \fg{(5,1,1,1,1,1)}{k} - \fg{(5,1,1,1,1,0)}{k}+ \fg{(5,2,2,1,1,0)}{k}.
\end{align*}
%
The coefficient of the first summand is 1, but
$$1\neq-1=(-1)^{|(5,2,2,1,1,1)|-|(5,1,1,1,1,1)|-1}.$$

\item The reason for the condition that $\lambda\in\spkl$ is strictly decreasing partition in Theorems~\mref{thm:aim2} and~\mref{thm:aim3} is as follows.
If $\lambda\notin\spkl$, then there exists $z\in[\ell-1]$ such that $\lambda_z=\lambda_{z+1}$. So we cannot apply
Proposition~\mref{prop:ind} in our method to arrive proofs of Theorems~\mref{thm:aim2} and~\mref{thm:aim3}.
\end{enumerate}
\end{remark}

\smallskip

\noindent
{\bf Acknowledgements}: This work is supported by National Natural Science
Foundation of China (12071191) and Innovative Fundamental Research Group Project of Gansu
Province (23JRRA684).

\noindent
{\bf Declaration of interests.} The authors have no conflicts of interest to disclose.

\noindent
{\bf Data availability.} Data sharing is not applicable as no new data were created or analyzed.

\end{document}